\documentclass[10pt,a4paper]{amsart}
\usepackage[utf8]{inputenc}
\usepackage{graphicx}
\usepackage{tikz}
\usepackage[normalem]{ulem}
\usepackage[bottom]{footmisc}

\usepackage{amsmath,amssymb,amsthm,mathtools}
\usepackage{mathabx}

\usepackage{hyperref}
\usepackage{enumerate}

\newtheorem{theorem}{Theorem}
\newtheorem*{theorem*}{Theorem}
\newtheorem{corollary}[theorem]{Corollary}
\newtheorem{lemma}[theorem]{Lemma}
\newtheorem{proposition}[theorem]{Proposition}
\newtheorem*{question*}{Question}

\theoremstyle{definition}

\newcommand\R{\mathbb{R}}

\newcommand\ES{\mathbb{S}}
\newcommand\Z{\mathbb{Z}}
\newcommand\N{\mathbb{N}}
\newcommand\C{\mathbb{C}}

\newcommand\T{\mathbb{T}}

\newcommand\cC{{\mathcal C}}
\newcommand\cD{{\mathcal D}}

\newcommand\cF{{\mathcal F}}
\newcommand\cG{{\mathcal G}}
\newcommand\cH{{\mathcal H}}
\newcommand\cK{{\mathcal K}}
\newcommand\cL{{\mathcal L}}

\newcommand\cS{{\mathcal S}}
\newcommand\cT{{\mathcal T}}
\newcommand\cU{{\mathcal U}}

\newcommand\cW{{\mathcal W}}

\newcommand\ud{\mathrm{d}}
\newcommand\Sconf{\mathcal{S}^{\mathrm{conf}}}
\newcommand\meanr{\bar{r}}
\DeclareMathOperator{\supp}{supp}

\newcommand\id{\mathrm{id}}
\newcommand\e{\varepsilon}

\newcommand\rank{{\hbox{\rm rank}}\,}

\begin{document}

\title[Conformal Hamiltonian flows]{THE DYNAMICS OF  CONFORMAL HAMILTONIAN FLOWS: DISSIPATIVITY AND CONSERVATIVITY}

\author[S. Allais]{Simon Allais$^{\circ\dag}$}
\author[M.-C. Arnaud]{Marie-Claude Arnaud$^{\circ\dag\ddag}$}
\address{Simon Allais, Universit\'e  Paris Cit\'e and Sorbonne Universit\'e, CNRS, IMJ-PRG, F-75006 Paris, France.}
\email{simon.allais@imj-prg.fr}
\urladdr{https://webusers.imj-prg.fr/~simon.allais/}
\address{Marie-Claude Arnaud, Universit\'e  Paris Cit\'e and Sorbonne Universit\'e, CNRS, IMJ-PRG, F-75006 Paris, France.}
\thanks{\noindent$\circ $ Universit\'e   Paris Cit\'e  and Sorbonne Universit\'e, CNRS, IMJ-PRG, F-75006 Paris, France \\
$\dag$ ANR AAPG 2021 PRC CoSyDy: Conformally symplectic dynamics, beyond symplectic dynamics, ANR-CE40-0014\\
$\ddag$ Member of the {\sl Institut universitaire de France}}

\email{marie-Claude.Arnaud@imj-prg.fr}
\date{}

\maketitle

\begin{abstract} We study in detail the dynamics of conformal Hamiltonian flows that are defined  on a conformal symplectic manifold (this notion was popularized by Vaisman in 1976). We show that they exhibit some conservative and dissipative behaviours. We also build many examples of various dynamics that show simultaneously their difference and resemblance with the contact and symplectic case.
\end{abstract}

\tableofcontents

\section{Introduction}
Symplectic dynamics models many conservative movements. Yet, other phenomena are dissipative and require another setting. This is the case of the damped mechanical systems: they are modelled by conformal Hamiltonian dynamics, which alter the symplectic form  up to a scaling factor.

This notion of  conformal symplectic dynamics can be placed in a broader context. To define such a dynamics, we only need to know in charts  an  equivalence class of 2-forms   for the relation   $\omega_1\sim \omega_2$,
\begin{center}
    where $\omega_1\sim \omega_2$ if $\omega_1=f\omega_2$ for some non-vanishing function $f$.\end{center}
    
A manifold endowed with such an equivalence class of local 2-forms, one of them being closed, is called a conformal symplectic manifold, a notion popularized  by Vaisman in \cite{Vaisman1976}. An equivalent notion  is  the notion of conformal structure $(M,\eta,\omega)$, a manifold $M$ endowed with a 1-form called the Lee form and a 2-form called the conformal form,  whose precise definition is given in section \ref{cccsm}. A proof of the equivalence of the two notions is given in \cite{ChantraineMurphy2019}. 

 We will study autonoumous conformal  Hamiltonian flows (CHF in short) $(\varphi_s)_{s\in\R}$ of compact manifolds, see definition in Section \ref{ssdefhamdyn}. They alter the conformal form up to a non-constant scaling factor. As the volume $\omega^n$ can increase or decrease at different points of the manifold under the action of the dynamics, we can expect different behaviours, some of them being {\sl conservative} e.g. completely elliptic periodic orbits, invariant foliations with compact leaves  and some other being {\sl dissipative}, e.g.   attractors or repulsors.

A precise definition of what we call conservative or dissipative requires the introduction of a notion related to the shape of the orbits. 
The {\sl  winding} of a point $x\in M$ through time is
defined as the map $t\mapsto r_t(x)$ ($r$ stands for ``rotation''),
\begin{equation*}
    r_t(x) := \int_0^t \eta(\partial_s\varphi_s(x))\ud s,\quad
    \forall t\in\R.
\end{equation*}
Then $\varphi_t^*\omega = e^{r_t}\omega$, see Lemma \ref{lem:rt}, and  a point $x\in M$ is
\begin{itemize}
    \item  either (positively) dissipative  when $\lim_{t\to+\infty} |r_t(x)|=+\infty$;
    \item or (positively) conservative.
\end{itemize}
Our main result, Proposition \ref{prop:CD}, asserts that for every CHF $(\varphi^H_t)$, if $\mathcal D_+$ is the set of positively dissipative points and if $\mathcal C_+$ the set of positively conservative points, then up to a set of zero volume, $\mathcal C_+$ coincides with the set of positively recurrent points, and then $\mathcal D_+$ with the set of positively non-recurrent points. Also, the $\omega$-limit set $\omega(x)$ of every $x\in\mathcal D_+$ is contained in $\{ H=0\}$.\\
Some   examples of conservative and dissipative points are
\begin{itemize}
    \item every attractor intersects $\{ H=0\}$, has non-trivial homology and almost every point in its basin of attraction that doesn't belong to the attractor is in $\mathcal D_+$, Corollary \ref{cor:attractorLee};
    \item if $x$ is a periodic points that is not a critical point of $H$, then 
    \begin{itemize}
        \item when $x\in\mathcal C_+$, the first return map to a Poincar\'e section preserves a closed 2-form and a foliation into (local) hypersurfaces;
        \item when $x\in \mathcal D_+$, then $H(x)=0$ and the first return map to a Poincar\'e section alters a certain closed 2-form up to a constant factor that is different from 1;
    \end{itemize}
    \item every fixed point of the flow is conservative.
\end{itemize}
We will provide also an example of wild conservative points: points that are recurrent, in $\{ H\neq 0\}$  but whose $\omega$-limit set intersects $\{ H=0\}$, see Section \ref{se:oscillating}. Hence these points satisfy  $\liminf_{t\to+\infty} |r_t(x)|=+\infty$. The origin of most of our examples is contact geometry. In particular, in Section \ref{sstwistedstructure}, we introduce a notion of twisted conformal symplectization that is crucial to the elaboration of examples and counter-examples.  \\
In a similar way, switching $H$ to $-H$, the set $\mathcal C_-$ of negatively conservative points is the set of $x\in M$ such that
$\lim_{t\to-\infty} |r_t(x)|=+\infty$ and $\mathcal D_-=M\backslash \mathcal C_-$ is the set of negatively dissipative points. We prove in Proposition \ref{prop:C_+} that $\cC_-$ and $\cC_+$  are always equal up to a set of zero volume. It is not true that for a general flow, the set of positively recurrent points is equal to the set of negatively recurrent points up to a set of volume zero.

 A CHF $(\varphi_t)$ is (positively) conservative when $\mathcal C_+=M$ and dissipative when $\mathcal C_+$ has zero volume.  We   highlight a strong relation between the topology of $\{ H=0\}$ and the property of being convervative:
 when $\{ H=0\}$ has a neighbourhood $V$ such that for every loop  $\gamma:\T\to V$, $\int_\gamma\eta=0$, then $(\varphi_t^H)$ is conservative, Section \ref{ssexactH=0}.  This   contains the case when $H$ doesn't vanish, Section \ref{ssHnonnul}.
    But there exist some examples of conservative CHF that are not in  this case, Section \ref{ssconsnonexact}. As the non-vanishing property is open in $C^0$-topology, we obtain $C^0$-open sets Hamiltonians $H$ such that the associated CHF flows are conservative.
    
    Among the conservative CHF, the Lee flows are those that correspond to the Hamiltonian $H=1$ for some choice of representative $(\eta,\omega)$ of the conformally symplectic structure. They are an extension of the Reeb flows in the contact setting. We will provide in every dimension examples of Lee flows 
    \begin{itemize}
        \item that are transitive, Section \ref{sstransLee}; this is different from the Hamiltonian symplectic case, where the level sets of $H$ are preserved;
        \item that have no periodic orbits, Section \ref{ssLeenoperiodic}; Weinstein conjecture in the contact setting and Arnol'd conjecture in the symplectic setting assert the existence of periodic orbits. This example emphasizes one difference between the CHF and the Reeb flows as well as the symplectic Hamiltonian flows.
    \end{itemize}
We will give a 2-dimensional example of Lee flow that is minimal (Section \ref{ssaxaconsdim2}), but we don't know if there is such an example in higher dimension.

We will give in Part \ref{ssLegendreattractors} of Section \ref{ssExampleattrcators} an example of dissipative CHF, with one normally hyperbolic attractor that is a Lagrangian submanifold, one normally hyperbolic repulsor that is also a Lagrangian submanifold and the remaining part of the manifold that is filled with heteroclinic connections. 
This gives a $C^1$-open set of CHF that are dissipative. See also Section \ref{ssaxadissdim2}.

There also exist $C^1$-open sets of CHF such that both $\mathcal C_+$ and $\mathcal D_+$ have positive volume. This happens when there is a normally hyperbolic periodic attractor    and one non degenerate local minimum of $e^\theta H$ where $\theta$ is a local primitive of $\eta$.\\

Another feature of the conformally symplectic dynamics is that they preserve isotropy (this is even a characterization of these dynamics). This is a common point with symplectic dynamics and contact dynamics. Therefore, we extend  or amend  some classical results for the invariant submanifolds of Hamiltonian flows. In Section \ref{ssinvdistr}, we prove that the CHF have a codimension 1 invariant foliation and explain in Section \ref{ssisotropy} the relation for a submanifold between being tangent to this foliation, being invariant and being isotropic (or coisotropic). This is reminiscent of Hamilton-Jacobi equation in the usual Hamiltonian setting.\\
We deduce that on a conformal cotangent bundle (see Section \ref{se:cotangent}), a Lagrangian invariant graph is necessarily  contained in the zero level set, which is a major difference with  the usual Hamiltonian setting.\\
Motivated by  the  result of Herman in the exact symplectic setting, \cite{Herman1989}, which asserts that every invariant torus on which the dynamics is $C^1$-conjugate to a minimal rotation is isotropic,  we  consider   tori $\mathcal T$ that are invariant by a CHF and such that the restricted dynamics is topologically conjugate to a rotation. In Section \ref{ssisotropy}, we recall the definition of the asymptotic cycle of an invariant measure and introduce in a similar way  the asymptotic cycle for flows on tori that are $C^0$-conjugate to a non necessarily minimal rotation. We prove that if the product of the cohomology class of the Lee form by the asymptotic cycle of $\mathcal T$ is non-zero, then $\mathcal T$ is isotropic. In particular, when the cohomology class of the Lee form is rational and when the rotation is minimal, the invariant torus is isotropic.

\subsection{2-dimensional examples}\label{ssexa2dim}

\subsubsection{A dissipative example}\label{ssaxadissdim2}

Let us discuss a simple
two-dimensional dissipative example that   illustrates some of our results.
Let $(M,\eta,\omega) = (\T^2,\ud x,\ud x\wedge\ud y)$ where $\T^2$
denotes the $2$-torus $\R^2/\Z^2$ and let $H:\T^2\to\R$
be the Hamiltonian function $H(x,y) = \sin(2\pi y)$.
We have pictured integral curves of the associated dynamics on Figure~\ref{fig:sinexample}.
\begin{figure}[!ht]
\centering
\includegraphics{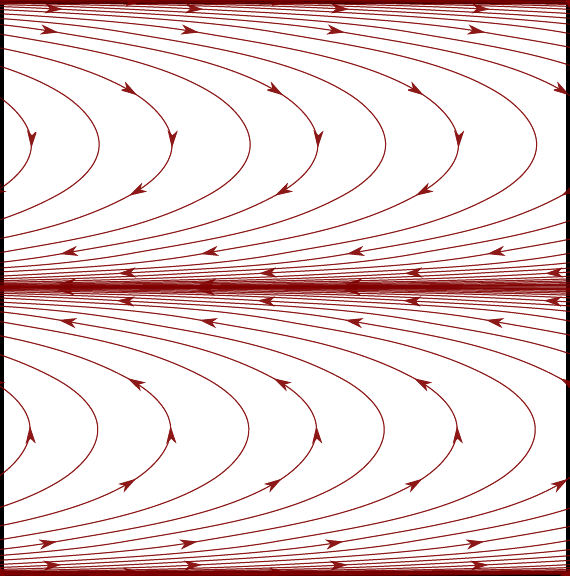}
\caption{Dynamics of $H(x,y)=\sin(2\pi y)$ in the
fundamental domain $[0,1]^2$}
\label{fig:sinexample}
\end{figure}
In this figure, we see that the only level set of $H$ that is preserved
is $\{ H=0\}$ and that it has two connected components:
an attractive circle and a repelling one.
Such a picture can be drawn in any dimension:
\emph{if the Lee form is not exact, there exist Hamiltonian flows
with attractive or repelling hyperbolic orbits} (\emph{cf}.
Proposition~\ref{prop:attractingPeriodicOrbits}).
Attractors (or repellers) are not necessarily contained in $\{ H=0\}$:
lifting this dynamics on the cover $\R/\Z\times\R/2\Z$, one
could see any of the two cylinders bounding the two attracting circles as
attractors.
However, as we will see attractors always intersect $\{ H=0\}$.
On Figure~\ref{fig:sinexample}, we see that the attractor is
winding in the $x$'s direction.
In general, \emph{the Lee form is not exact in any neighborhood
of the intersection of an attractor with $\{ H=0\}$}
(\emph{cf.} Corollary~\ref{cor:attractorLee}).
In particular, an attractor cannot be finite and must intersect $\{ H=0\}$.

\subsubsection{A conservative example}\label{ssaxaconsdim2}
In the opposite direction, let us point out the existence of
conformal Hamiltonian dynamics that preserve the symplectic
form $\omega$ but the behavior of which nonetheless differs from
the symplectic Hamiltonian case.
As a simple 2-dimensional example, let us consider the 2-torus
$\T^2$ endowed with its canonical area form $\omega=\ud x\wedge\ud y$
once again. Let us fix $a,b\in\R$ and choose the Lee form
$\eta := a\ud x+b\ud y$.
The Hamiltonian flow of $H\equiv 1$, which is called
the Lee flow associated to the representative
of the conformally symplectic structure
(the \emph{gauge}) $(\eta,\omega)$,
is $\varphi_t(x,y)=(x+bt,y-at)$.
If $a$ and $b$ are rationally independent, this flow is
minimal.
This is a striking difference with autonomous Hamiltonian
flows of symplectic manifolds, where trajectories are never
dense and there usually are plenty of periodic orbits.
In general, we prove that
\emph{there exist topologically transitive Lee flows in
any dimension} and that
\emph{there exist Lee flows without periodic orbit in
any dimension} (\emph{cf}. Propositions~\ref{prop:transitiveLee}
and \ref{prop:aperiodicLee}).
In both cases, the Lee form $\eta$ is not completely resonant  
  (\emph{i.e.} the set of its integrals along the loops  is a dense subgroup of $\R$),
which is necessary in order to have dense trajectories.
One could ask whether there always is a periodic orbit
when $\eta$ is  completely resonant,  
but this question is harder than
solving the Weinstein conjecture (\emph{i.e.} the existence of
a periodic orbit for any Reeb flow of a closed contact manifold).
 \subsection{ Structure of the article}
 \begin{itemize}
     \item In section \ref{SPrem}, we introduce the notions of conformal manifold and conformal Hamiltonian dynamics and prove some of their properties, Then we give some examples: the conformal cotangent bundle, the twisted conformal symplectization, and describe the invariant foliation.
     \item In section \ref{se:CvsD}, we characterize the global conservative-dissipative decomposition of the dynamics in term of recurrence. We prove the almost everywhere coincidence of the behaviours in the past and in the future. We also prove that the boundedness of the winding number implies the existence of invariant measures. We also provide an example of orbits that are conservative and have a strange oscillating behaviour.
     \item In section \ref{se:etaExact}, we give some topological conditions on $\{ H=0\}$ that imply that the dynamics is conservative. We gave some examples of such dynamics that are transitive and some others that have no periodic orbit. We   give an example of a conservative dynamics for which the topological condition for $\{ H=0\}$ is not satisfied.
     \item In section \ref{Sdiss}, we begin by studying some ergodic measures whose support is dissipative. Then we give examples of dissipative dynamics with Lagrangian attractors and repulsors, and also examples with periodic attractors and repulsors. We also give sufficient conditions implying that some connected component of $\{ H=0\}$ cannot be an atttractor.
     \item In section \ref{Sdistsub}, we give some condition that implies that a component of $\{ H=0\}$ is in the closure of a non-compact leaf of the invariant distribution. Then we study invariant submanifolds from different points of view: their position relatively to the invariant foliation, and when they are rotational tori, the relations between their asymptotic cycle and their isotropy.
     \item Finally, there is an appendix dealing with isotropic submanifolds.
 \end{itemize}
 
\subsection {  Acknowledgements.} {  The authors are grateful to Ana Rechtman for listening some  preliminary versions of this work, discussing them  and   pointing out the link with   classical results on foliations.}
The first author was supported by the postdoctoral fellowship of the Fondation Sciences
Mathématiques de Paris.

\section{Preliminaries}\label{SPrem}

\subsection{Conformal symplectic manifolds}\label{cccsm}

Given a closed 1-form $\eta$, the associated Lichnerowicz-De Rham differential
$\ud_\eta$ is defined on the differential forms $\alpha$
by $\ud_\eta \alpha := \ud\alpha - \eta\wedge\alpha$.
It satisfies $\ud_\eta^2 = 0$ and
$\ud_{\eta +\ud f}\alpha = e^f \ud_\eta (e^{-f}\alpha)$.
If $\ud_\eta \alpha = 0$, one says that $\alpha$ is $\eta$-closed.

Given an even dimensional manifold $M$,
a conformal symplectic structure is an equivalence class
of couples $(\eta,\omega)$ where $\eta$ is a closed 1-form
of $M$ and $\omega$ is a non-degenerate 2-form that
is $\eta$-closed, two such couples $(\eta_i,\omega_i)$,
$i\in\{1,2\}$, being equivalent if
there exists a map $f:M\to\R$ such that
$\eta_2 = \eta_1 + \ud f$ and $\omega_2 = e^f\omega_1$.
A conformal symplectic manifold is an even dimensional manifold $M$
endowed with a conformal symplectic structure,
we will often work with a specific representative $(\eta,\omega)$
and write $(M,\eta,\omega)$ the conformal symplectic manifold.
A notion that does not depend on the specific choice of
representative $(\eta,\omega)$ is called gauge invariant or
well defined up to gauge equivalence.
The closed 1-form $\eta$ is called the Lee form of $(M,\eta,\omega)$,
its cohomology class $[\eta]\in H^1(M;\R)$ is gauge invariant.
A conformal symplectomorphism $\varphi : (M_1,\eta_1,\omega_1)
\to (M_2,\eta_2,\omega_2)$ is a diffeomorphism $\varphi:M_1\to M_2$
such that $\varphi^*\eta_2 = \eta_1 + \ud f$ and
$\varphi^*\omega_2 = e^f \omega_1$ for some $f:M_1\to\R$
(this notion is gauge invariant).  When $\dim M\geq 4$, the second equality implies the first one.

Similarly to the symplectic case,
a submanifold $N$ of a conformal symplectic manifold $(M,\eta,\omega)$
is called isotropic if $TN \subset TN^\omega$,
coistropic if $TN^\omega\subset TN$ and lagrangian if
$TN=TN^\omega$ (where $E^\omega$ denotes the $\omega$-orthogonal bundle
of the bundle $E$), this notion is gauge invariant.

A symplectic manifold $(M,\omega)$ has a natural conformal symplectic
structure $(0,\omega)$ (which is the same as $(0,\lambda\omega)$
for $\lambda\in\R^*$); conversely, a conformal symplectic structure
$(\eta,\omega)$ comes from a symplectic structure if and only if
$\eta$ is exact.

\subsection{Hamiltonian dynamics}\label{ssdefhamdyn}

Given a map $H:M\to\R$ defined on a conformal symplectic
manifold $(M,\eta,\omega)$, we define its associated Hamiltonian
vector field $X$ by $\iota_X \omega = \ud_\eta H$,
conversely $H$ is the Hamiltonian of $X$.   When the cohomology class of $\eta$ is not $0$, then $H$ is unique.
 This matching    Hamiltonian-vector field  does depend on the choice of representative $(\eta,\omega)$
but not the algebra of Hamiltonian vector fields:
the previous vector field $X$ is the same as the one induced by $e^f H$
for the Lee form $\eta+\ud f$.
One can extend this definition to time-dependent Hamiltonian maps
but we will focus on autonomous Hamiltonian in this paper.
When $H\equiv 1$, the associated vector field $L^\eta$
is called the Lee vector field of $\eta$ and
its flow is called the Lee flow.

Let us assume that the vector field $X$ associated with $H$ is complete.
Let us denote $(\varphi_t)$ its flow and
\begin{equation*}
    r^H_t(x) := \int_0^t \eta(X\circ\varphi_s(x)) \ud s,\quad
    \forall x\in M,\forall t\in\R.
\end{equation*}
When the choice of $H$ is clear, we set $r_t := r_t^H$.
\begin{lemma}\label{lem:rt}
    Given a complete Hamiltonian flow $(\varphi_t)$ on $(M,\eta,\omega)$
    associated with a Hamiltonian
    $H$, for all $t\in\R$,
    \begin{equation*}
        \varphi_t^*\omega = e^{r_t}\omega,\
        \varphi_t^*\ud_\eta H = e^{r_t}\ud_\eta H,\
        H\circ\varphi_t = e^{r_t} H
        \text{ and }
        \varphi_t^*\eta = \eta + \ud r_t.
    \end{equation*}
    In particular, the level set $\{H=0\}$ is invariant
    under the flow and $\frac{1}{H}\omega$ is an invariant 2-form on $\{H\neq 0\}$.
\end{lemma}

\begin{proof}
    By taking the Lie derivative of $\omega$,
    \begin{equation*}
        \cL_X \omega = \ud(\ud_\eta H) + \iota_X(\eta\wedge\omega)
        = \eta\wedge\ud H + \eta(X)\omega
        - \eta\wedge(\iota_X\omega)
    \end{equation*}
    Thus,
    \begin{equation*}
        \cL_X \omega = \eta\wedge\ud H + \eta(X)\omega
        - \eta\wedge\ud H + \eta\wedge\eta H
        = \eta(X)\omega,
    \end{equation*}
    which implies the first equality of the statement.
    
    We deduce that 
   \begin{eqnarray*}\varphi_t^*(\ud_\eta H)&=\varphi_t^*(\iota_X\omega)&=\omega(X\circ \varphi_t, \ud\varphi_t\cdot)=\omega(\ud\varphi_t X, \ud\varphi_t\cdot)\\
    &=\iota_X(\varphi_t^*\omega)&=\iota_X(e^{r_t}\omega)=e^{r_t}\ud_\eta H.
 \end{eqnarray*}
    Injecting $X$ in $\iota_X\omega = \ud_\eta H$,
    one gets $\ud H\cdot X = \eta(X)H$, which   implies that
    $h(t):=H\circ\varphi_t(x)$, for a fixed $x\in M$,
    satisfies $h'(t)=\eta(X\circ\varphi_t(x))h(t)$ and the third statement follows.
    Finally, the last statement is due to $\cL_X \eta = \ud(\eta(X))$.
\end{proof}
We remark that the relations $\varphi^*_t\omega=e^{r_t}\omega$
and $\varphi_t^*\eta = \eta + \ud r_t$ are also
satisfied in the time-dependent setting.
This indeed implies that conformal Hamiltonian diffeomorphisms
are conformal symplectomorphisms.

\subsection{Conformal cotangent bundles}\label{se:cotangent}
Given a manifold $L$ endowed with a closed $1$-form $\beta$,
one can define a conformal symplectic structure on $T^*L$
denoted $T^*_\beta L$ in the following way.
Let $\pi:T^*L\to L$ be the cotangent bundle map
and $\lambda$ the associated Liouville form:
$\lambda_{(q,p)}\cdot\xi := p(\ud\pi\cdot\xi)$.
The conformal structure defining $T^*_\beta L$
is $(\eta,\omega):=(\pi^*\beta,-\ud_\eta\lambda)$.
The neighborhood of the $0$-section of $T^*_\beta L$ is
a model of a neighborhood of a Lagrangian embedding of $L$
pulling back the Lee form to $\beta$
(see Section~\ref{se:weinstein}).

Let us recall how one can canonically extend diffeomorphisms
and flows of $M$ to conformal symplectomorphisms and Hamiltonian flows
of $T^*_\beta M$.
Let $f:M\to N$ be a diffeomorphism, one can symplectically extend
it to $\hat{f}:T^*M\to T^*N$ by the well-known formula:
\begin{equation*}
    \hat{f}(q,p) = \left(f(q),p\circ \ud f_q^{-1}\right),\quad
    \forall (q,p)\in T^*M.
\end{equation*}
Now if the diffeomorphism $f:M\to N$ satisfies $f^*\beta = \alpha +\ud r$,
for closed $1$-forms $\alpha$, $\beta$ and some map $r:M\to\R$,
the extension $\hat{f}:T^*_\alpha M \to T^*_\beta N$ defined by
\begin{equation*}
    \hat{f}(q,p) = (f(q),e^{r(q)} p \circ\ud f_q^{-1}),\quad
    \forall (q,p)\in T^*M
\end{equation*}
is conformally symplectic.
Indeed, let us denote by $\pi_M$, $\pi_N$ the associated cotangent bundle maps,
$\lambda_M$, $\lambda_N$ the associated Liouville forms.
Then $\hat{f}^*\lambda_N = e^{r\circ\pi_M}\lambda_M$:
\begin{equation*}
    \left(\hat{f}^*\lambda_N\right)_{(q,p)}\cdot\xi =
    e^{r(q)}p\circ\ud f^{-1} \circ\ud\pi_N\circ\ud \hat{f} \cdot\xi
    = e^{r(q)}p\circ \ud\pi_M\cdot\xi
\end{equation*}
as $\pi_N \circ\hat{f}=f\circ\pi_M$.
We deduce $\hat{f}^*(\ud_{\pi_N^*\beta}\lambda_N) = 
e^{r\circ\pi_M}\ud_{\pi_M^*\alpha}\lambda_M$:
\begin{equation*}
    \begin{split}
    \hat{f}^*(\ud\lambda_N - \pi_N^*\beta\wedge\lambda_N)
    &= \ud(e^{r\circ\pi_M}\lambda_M) - \pi_M^*(\alpha+\ud
    r)\wedge(e^{r\circ\pi_M}\lambda_M)\\
    &= e^{r\circ\pi_M}(\ud \lambda_M - \pi_M^*\alpha\wedge\lambda_M).
\end{split}
\end{equation*}

Now given a flow $f_t : M\to M$, with $f_0=\id$, of associated vector field
$X_t$, one has $f_t^*\beta = \beta + \ud r_t$ with
$r_t(q) := \int_0^t \beta(X_s\circ f_s(q))\ud s$
so the associated conformal symplectic flow $(\hat{f}_t)$ is
well-defined and one checks that it corresponds to the Hamiltonian flow
of $H_t(q,p) = p(X_t(q))$.

\subsection{Twisted conformal symplectizations}\label{sstwistedstructure}
A large class of conformal symplectic manifold that are non-symplectic
is given by the conformal symplectizations of contact manifolds.
Let $(Y^{2n+1},\alpha)$ be a manifold endowed with a contact form $\alpha$
(\emph{i.e.} a 1-form satisfying $\alpha\wedge(\ud\alpha)^n\neq 0$),
its conformal symplectization $\Sconf(Y,\alpha)$ is the manifold
$Y\times S^1$ endowed with the structure $(\eta=-\ud\theta,\omega =
-\ud_\eta(\pi^*\alpha))$
where $S^1=\R/\Z$ whereas $\theta:Y\times S^1\to S^1$ and $\pi:Y\times S^1\to Y$
are the canonical projections.
  The conformal symplectization only depends on the oriented contact
distribution $\ker\alpha$.
Indeed, when $(\eta',\omega') = (\eta -\ud f,
-\ud_{\eta'}(e^{-f}\alpha))$,
$(x,\theta)\mapsto (x,\theta-f(x))$
is a conformally symplectic diffeomorphism between   $(\omega,\eta)$ and $(\omega',\eta')$.

Given a closed 1-form $\beta$ of $Y$, we also define the $\beta$-twisted conformal
symplectization of $(V,\alpha)$ by replacing $\eta$ in the
previous definition with $\eta = \pi^*\beta - \ud\theta$,
we denote it $\Sconf_\beta(Y,\alpha)$.
We check that $\omega$ is non-degenerate by showing that $\omega^{n+1}$
does not vanish:
\begin{equation*}
    (-1)^{n+1}\omega^{n+1} = (n+1)(\ud\theta - \pi^*\beta)\wedge\pi^*\alpha
    \wedge (\ud(\pi^*\alpha))^n
    = (n+1)\ud\theta\wedge \pi^*(\alpha\wedge(\ud\alpha)^n) \neq 0,
\end{equation*}
the second equality comes from the fact that
$\beta\wedge\alpha\wedge(\ud\alpha)^n = 0$ for a degree reason
and the contact hypotheses implies the non-vanishing of the last
expression.
When the choice of the contact form $\alpha$ is clear,
the couple $(\eta,\omega)=(\pi^*\beta-\ud\theta,-\ud_\eta(\pi^*\alpha))$
as well as the associated Lee vector field and Hamilton equations
will be implicitly chosen or referred to as standard.

Let us show how the study of conformal Hamiltonian
dynamics will also inform us about contact Hamiltonian
dynamics , see also Proposition \ref{PLegendrianattractors}.
We recall that
the contact Hamiltonian vector field $X$ associated with the contact Hamiltonian
map $H:Y\to\R$ is defined by
\begin{equation}\label{eq:contactHam}
    \begin{cases}
        \alpha(X) = H,\\
        \iota_X\ud\alpha = (\ud H\cdot R)\alpha -\ud H,
    \end{cases}
\end{equation}
where $R$ is the Reeb vector field defined by $\alpha(R)=1$ and
$\iota_R\ud\alpha = 0$ (the Hamiltonian vector field associated with $H=1$).
\begin{lemma}\label{lem:HamiltonSconf}
    Let $H:Y\to\R$ be a contact Hamiltonian map of the contact manifold $(Y,\alpha)$
    with fixed contact form $\alpha$ associated with the Reeb vector field $R$
    and let $X$ be the associated Hamiltonian vector field.
    The conformal Hamiltonian vector field on $\Sconf_\beta(Y,\alpha)$
    associated with $\widetilde{H}:(x,\theta)\mapsto H(x)$ is
    \begin{equation*}
        \widetilde{X} =
        X\oplus (\beta(X)-\ud H\cdot R)\partial_\theta \in TY\oplus TS^1.
    \end{equation*}
    In particular, the standard Lee vector field is
    $L:=R\oplus\beta(R)\partial_\theta$.
\end{lemma}

\begin{proof}
    Let us first derivate the expression of the Lee vector field
    $L$.
    Let $\eta := \pi^*\beta- \ud\theta$ be the Lee form and
    $\omega := -\ud_\eta (\pi^*\alpha)$ be associated symplectic form.
    Let us write $L=V\oplus f\partial_\theta$, $V$ being a vector field
    of $Y$ and $f:Y\to\R$.
    Since $\iota_L\omega = -\eta$, one has $\eta(L) = 0$, that is
    $f = \beta(V)$.
    Developing the Lee equation, one then gets
    \begin{equation*}
        \iota_L\ud(\pi^*\alpha) + \alpha(V)(\pi^*\beta - \ud\theta)
        = \pi^*\beta - \ud\theta.
    \end{equation*}
    By identification, $\alpha(V)=1$ and $\iota_V\ud\alpha = \beta - \alpha(V)\beta
    = 0$,
    therefore $V=R$.
    The general case is also deduced by identification, once we have remarked that
    \begin{equation*}
        \eta(\widetilde{X}) = \omega(\widetilde{X},L) =
        \ud \widetilde{H}\cdot L - \widetilde{H}\eta(L) = \ud H\cdot R.
    \end{equation*}
\end{proof}
Therefore the conformal Hamiltonian flow $(\Phi_t)$ of $\Sconf_\beta(Y,\alpha)$
lifting the contact Hamiltonian flow $(\varphi_t)$ is
$\Phi_t(x,\theta) = (\varphi_t(x),\theta + \rho_t(x) - r_t(x))$
where
\begin{equation*}
    r_t(x) =
    r_t^{\widetilde{H}}(x,\theta) = \int_0^t (\ud H(\varphi_s(x))\cdot R)
    \ud s \quad \text{and}\quad
    \rho_t(x) = \int_0^t \beta(\partial_s\varphi_s(x))\ud s.
\end{equation*}
The expression of $r_t$ is consistent with the following general fact for
conformal Hamiltonian vector fields $X$ of $(M,\eta,\omega)$:
\begin{equation*}
    \eta(X) = \omega(X,L) = \ud H\cdot L -
    \eta(L)H = \ud H\cdot L,
\end{equation*}
where $L$ is the Lee vector field.
An isotropic embedding $i:L\hookrightarrow  (Y,\alpha)$
is by definition an embedding such that $i^*\alpha=0$,
it is Legendrian when the dimension of $L$ is maximal:
$2\dim L + 1 = \dim Y$.
One can associate to every isotropic submanifold $L\subset Y$
the isotropic lift $L\times S^1\subset \Sconf_\beta (Y,\alpha)$.
Therefore, dynamical properties of contact Hamiltonians
can be deduced from properties of conformal Hamiltonians
``by projection $\Sconf_\beta (Y,\alpha)\to Y$''.   See Part \ref{ssLegendreattractors} of section \ref{ssExampleattrcators}.

\subsection{The invariant distribution $\cF$}\label{ssinvdistr} 

In the conformal setting, the Hamiltonian map $H$ is not
an integral of motion.
But the (singular) distribution $\cF:=\ker\ud_\eta H$ is still
invariant since $\varphi^*_t\ud_\eta H = e^{r_t}\ud_\eta H$
(Lemma~\ref{lem:rt}). 
Moreover, we have
$$\ud\big(\ud_\eta H)=\eta\wedge \ud H=\eta\wedge \ud_\eta H
$$
hence by Frobenius theorem, at every regular point the Pfaffian distribution $\ker\ud_\eta H$ is integrable.

However, in dynamical systems with dissipative behaviors,
its regular leaves are often non-compact
(the important exception being $\{ H=0\}$).
Let us describe the major properties of $\cF$.

\begin{lemma}\label{lem:Hpath}
If $\gamma:[0,1]\to M$ is a path tangent to $\cF$,
then $H(\gamma(1)) = e^{\int_\gamma \eta} H(\gamma(0))$.
\end{lemma}

\begin{proof}
    Similarly to the proof of Lemma~\ref{lem:rt},
    $h:=H\circ\gamma$ satisfies $h'=\eta(\dot{\gamma})h$.
\end{proof}

\begin{corollary}\label{cor:01law}
    Every connected submanifold $L\subset M$ tangent to $\cF$ is either
    included in $\{ H =0\}$ or in $\{ H\neq 0\}$.
    In the case where the pull-back of the Lee form to $L$ is not
    exact, $L$ is included in $\{ H=0\}$.
\end{corollary}

In the symplectic case, regular levels of $H$ admit invariant
volume forms (see \emph{e.g.} \cite[\S I.8]{Ekeland1990}),
the following proposition generalizes this
phenomenon.

\begin{proposition}\label{prop:leafvolume}
    Let $(M,\eta,\omega)$ be a $2n$-dimensional closed conformal symplectic
    manifold and $H:M\to\R$ a Hamiltonian, the
    flow of which is $(\varphi_t)$.
    Let $i:\Sigma\hookrightarrow M$ be an embedded leaf tangent to $\cF$.
    Then there exists a volume form $\mu$ of $\Sigma$
    such that $\varphi^*_t \mu = e^{(n-1)r_t}\mu$.
    Moreover, there exists a $(2n-1)$-form
    $\mu_0$ on $M$ such that $\mu=i^*\mu_0$ and
    $\mu_0\wedge \ud_\eta H = \omega^n$ in the neighborhood
    of $\Sigma$.
\end{proposition}

\begin{proof}
    Let $\mu_1$ be a $(2n-1)$-form on $M$ such that $i^*\mu_1$
    is a volume form of $\Sigma$ (which is oriented by $\ud_\eta H$).
    By assumption, $(\ud_\eta H)_x\neq 0$ for $x\in\Sigma$ whereas
    $i^*\ud_\eta H = 0$ so $(\mu_1\wedge\ud_\eta H)_x \neq 0$
    for every $x\in\Sigma$.
    Since $\Sigma$ is an embedded leaf, there exists
    an open neighborhood $U$ of $\Sigma$ on which
    $\mu_1\wedge\ud_\eta H$ does not vanish.
    There exists $f:M\to\R$ that does not
    vanish on $U$ such that $f\mu_1\wedge\ud_\eta H = \omega^n$
    restricted to $U$.
    Let us show that
    $\mu_0:=f\mu_1$ and the volume form $\mu:=i^*\mu_0$ are
    the desired forms.

    Let us recall that $\cL_X\omega=\eta(X)\omega$
    and $\cL_X\ud_\eta H = \eta(X)\ud_\eta H$
    (Lemma~\ref{lem:rt}).
    Let us apply $\cL_X$ to the equation
    $\mu_0\wedge \ud_\eta H = \omega^n$:
    \begin{equation*}
        (\cL_X\mu_0)\wedge\ud_\eta H
        + \mu_0\wedge \eta(X)\ud_\eta H =
        n\eta(X)\omega^n.
    \end{equation*}
    Therefore,
    \begin{equation*}
        (\cL_X\mu_0)\wedge\ud_\eta H =
        (n-1)\eta(X)\mu_0\wedge \ud_\eta H
    \end{equation*}
    so that
    \begin{equation*}
        i^*(\cL_X\mu_0) = (n-1)\eta(X)i^*\mu_0.
    \end{equation*}
    Since the flow $(\varphi_t)$ preserves $\Sigma$,
    $i^*(\cL_X\mu_0) = \cL_X(i^*\mu_0)$ and
    the conclusion follows.
\end{proof}

\begin{corollary}
    Every embedded leaf of $\cF$
    outside $\{ H=0\}$ admits an invariant
    volume form
\end{corollary}

\begin{proof}
    Let $\mu$ be the volume form associated
    with $\Sigma$ by Proposition~\ref{prop:leafvolume}.
    The volume form $\frac{\mu}{H^{n-1}}$ is invariant.
\end{proof}
  When the embedded leaf of $\cF$ is not compact, this invariant volume can be unbounded.

\section{A global decomposition of the phase space: conservative versus dissipative}
\label{se:CvsD}
 Let us introduce three   notions of attractors that will be used in different parts  of this article.

 \begin{itemize}
    \item  
An invariant  compact subset $A\subset M$ is a {\em weak attractor}
if there exists an open subset $U\supset A$,
called a basin of attraction of $A$, such that
$\bigcup_{x\in U} \omega(x) \subset A$   where $\omega(x)$ is
the omega-limit set of $x$.   The basin of attraction is not necessarily unique.
\item A subset $A\subset M$ is a {\em strong attractor}
if there exists an open subset $U\supset A$,
such that $\forall t>0, \varphi_t(\overline{U})\subset U$ and $A = \bigcap_{t>0} \varphi_t(U)$
(which implies that $A$ is compact and invariant).
\item an invariant  closed submanifold $N\subset M$ is {\em normally hyperbolically attractive} if there exists a tubular neighbourhood $V=j(N\times[-\varepsilon_0, \varepsilon_0])$ of $N$  where $j:N\times[-\varepsilon_0, \varepsilon_0]\hookrightarrow M$ 
is an embedding, $\tau>0$ and $a\in (0, 1)$ such that $\varphi_\tau^H(V)\subset {\rm Int}(V)$ and if we denote $V_\epsilon=j(N\times[-\varepsilon, \varepsilon])$, then
$$  \forall \varepsilon\in (0, \varepsilon_0], \varphi_{\tau}^H(V_\varepsilon)\subset V_{a\varepsilon}.
$$
\end{itemize}
Observe that a strong attractor is always a weak attractor.

\subsection{The conservative-dissipative decomposition}
\label{se:CD}

  We defined in the introduction the partition in invariant sets $M=\cC_+\sqcup\cD_+$ with 
\begin{equation}\label{eq:C}
    \cC_+ := \left\{ x\in M\ |\ \liminf_{t\to+\infty} 
    |r_t(x)| < +\infty \right\}
    = \left\{ x\in M\ |\ \liminf_{p\to+\infty} 
    |r_p(x)| < +\infty \right\}
\end{equation}
and 
\begin{equation}\label{eq:D}
    \cD_+ := \left\{ x\in M \ |\   \lim_{t\to +\infty}
    |r_t(x)| = +\infty \right\}.
\end{equation}
The second definition of $\cC_+$ is to be understood with $p\in\N$;
the equality between both definitions is due to
$|\partial_t r_t(x)|\leq \|\eta(X)\|_\infty <+\infty$
(see Lemma~\ref{lem:rt}).

\begin{proposition}\label{prop:CD}   Up to a set with zero Lebesgue measure, the set of positively recurrent points coincides with $\cC_+$. The $\omega$-limit set  of every point   in $\cD_+$ is in $\{ H=0\}$. Almost every point in $\cD_+$ is in $\{ H\neq 0\}$ and if $x\in \cD_+\cap \{ H\neq 0\}$,  $r_t(x)\to -\infty$ as $t\to+\infty$ and  every neighbourhood of     $\omega(x)=A$ contains a closed curve $\gamma$ such that $\int_\gamma\eta\neq0$. Hence A is infinite.\\
    Moreover, for every embedded leaf $\Sigma$
    included in $\{ H \neq 0\}$ with a proper inclusion
    map,   up to a set with zero  $(n-1)$-dimensional volume,  
    $\cC_+\cap\Sigma$ coincides with the  set of positively recurrent points in $\Sigma$. 
\end{proposition}

\begin{proof}
    Let us first remark that $\omega^n$-almost every point
    of $\{ H=0 \}$ is trivially recurrent   and in $\cC_+$:
     every point of  the subset $\{ H= 0\}\cap \{ \ud H = 0\}$
    is fixed by the dynamics whereas
    $\{ H= 0\}\cap \{ \ud H \neq 0\}$ is
    negligible. 
    
    Since $H\circ\varphi_t = e^{r_t}H$ by  Lemma \ref{lem:rt},  a point $x\in M$
    satisfying $r_t(x)\to +\infty$ must be in
    $\{ H=0\}\cap\{ \ud H \neq 0\}$ which is a negligible
    set. Hence if $x\in \cD_+\cap \{ H\neq 0\}$,  $r_t(x)\to -\infty$ as $t\to+\infty$ and then $\lim_{t\to\infty} H(\varphi_tx)=0$ and $x$ is not positively recurrent.
    
    Let us now show that almost every point of
    $\cC'_+:=\cC_+\cap\{ H\neq 0\}$ is recurrent.
    According to Lemma~\ref{lem:rt},
    $H\circ \varphi_t = e^{r_t} H$, so
    \begin{equation*}
    \cC'_+ = 
    \left\{ x\in M\ |\ \limsup_{p\to+\infty} |H(\varphi_p(x))| \neq 0\right\}.
    \end{equation*}
    For $k\in\N^*$, let us define the following compact sets
    \begin{equation*}
        \cH_k := \left\{ x\in M\ |\ |H(x)|\geq \frac{1}{k} \right\}.
    \end{equation*}
    Then $\cC'_+$ is the increasing union of the $\cC'_k$'s defined by
    \begin{equation*}
        \cC'_k := \cH_k \cap \bigcap_{N\in\N}\bigcup_{p\geq N}
        \varphi_p^{-1}(\cH_k),\quad
        \forall k\in\N^*.
    \end{equation*}
    For each $k\in\N^*$, there is a well-defined first-return measurable map
    $f_k:\cC'_k\righttoleftarrow$, $f_k(x):= \varphi_{n(x)}(x)$ where
    $n(x):=\min \{ p\in\N^*\ |\ \varphi_p(x)\in \cH_k\}$.
    Since the 2-form $\frac{\omega}{H}$ of $\{ H\neq 0\}$
    is preserved by $\varphi_p$
    for all $p\in\N$ (Lemma~\ref{lem:rt}),
    the measurable maps $f_k$'s are preserving the measure $\nu:A\mapsto \int_A
    \frac{\omega^n}{H^n}$.
    Since, for $k\in\N^*$, $\cC'_k$ has a countable basis of open sets and a measure
    $\nu(\cC'_k)\leq \nu(\cH_k)\leq k^n\omega^n(M)$ which is finite,
    the Poincaré's recurrence theorem implies that almost every
    point of $\cC'_k$ is recurrent for $f_k$.\\
    Let us prove that if $x\in \cD_+\cap \{ H\neq 0\}$, every neighbourhood $V$ of $\omega(x)$ contains a closed curve $\gamma$ such that $\int_\gamma \eta\neq 0$. By compacity of $\omega(x)$,
    one can assume that $V$ is a finite union of
    path-connected contractible open sets $V_j$.
    Let $K>0$ be such that
    $|\int_\gamma \eta |< K$
    for every $\gamma :[0,1]\to V_j$ and every $j$
    (where $\eta$ denotes the Lee form). Let $T>0$ be such that for all $t\geq T$, $\varphi_t(x)\in V$ and 
   let $j_0$ be such that there exist arbitrarily
    large $t$'s satisfying $\varphi_t(x)\in V_{j_0}$.
    Let $t_1> t_0>T$ be such that $r_{t_0}(x) - r_{t_1}(x) > K$
    and $\varphi_{t_i}(x)\in V_{j_0}$ for $i\in\{ 0,1\}$.
    Then concatenating $t\mapsto \varphi_t(x)$,
    $t\in [t_0,t_1]$, with a path of $V_{j_0}$ connecting
    $\varphi_{t_1}(x)$ to $\varphi_{t_0}(x)$,
    one gets a loop $\gamma:I\to V$ satisfying
    $\int_\gamma \eta \neq 0$.
    The conclusion follows.
    
    Finally, let $\Sigma\subset\{ H\neq 0\}$
    be an embedded leaf of $\cF$ with a proper
    inclusion map.   We have seen that no point in $\cD_+\cap \Sigma$ is  positively recurrent.
    Since $\cC_+\cap\Sigma = \cC_+'\cap\Sigma$,
    it is enough to
    prove that almost every point of $\cC'_k\cap\Sigma$
    is recurrent, for all $k\in\N^*$.
    Let $\mu$ be the volume form associated with $\Sigma$
    by Proposition~\ref{prop:leafvolume},
    then the first return maps $f_k|_{\Sigma\cap\cC'_k}$
    are preserving the measure
    $\nu_\Sigma:A\mapsto \int_A \frac{\mu}{H^{n-1}}$
    defined on $\Sigma$.
    Since $\Sigma\hookrightarrow\{ H\neq 0\}$
    is proper, the $\Sigma\cap\cC'_k$'s
    are compact, so the $\nu_\Sigma(\Sigma\cap\cC'_k)$'s
    are finite.
    The conclusion follows.
\end{proof}
We recall that  $U\subset M$ is a  wandering set if 
$\exists T>0$,
 $\forall 
t\geq T, \varphi_t(U)\cap  U=\emptyset$.
\begin{corollary}\label{Cwandering}
Let $U$ be a wandering set. Then   almost every point of $U$
belongs to $\cD_+\cap\cD_-$ and satisfies $\lim_{t\to+\infty}H(\varphi_t^H(x))=\lim_{t\to-\infty}H(\varphi_t^H(x))=0$.
\end{corollary}

\begin{corollary}\label{cor:attractorLee}
    Let $A$ be a weak attractor with basin $U$,
    then for almost every point $x$ of $U\setminus A$,
    $r_t(x)\to -\infty$ as $t\to+\infty$.
    In particular, the Lee form is not exact in any neighborhood
    of ${A}\cap \{ H = 0\}$.
\end{corollary}

As a consequence, an attractor (or repeller) of $(\varphi_t)$ is never
a finite set.

\begin{proof}[Proof of Corollary \ref{cor:attractorLee}]
    By definition, points of $U\setminus A$ are not recurrent
    so almost every point of $U\setminus A$ is in $\cD_+$
    by Proposition~\ref{prop:CD}. The same proposition implies the other results.
\end{proof}

\subsection{Almost everywhere coincidence of past and future}
  In the remaining of the article, we will say that a property is almost everywhere satisfied with reference to every volume form. 
 We have of course that $\cC_-$ coincides with the set of negatively recurrent points. What is surprising is that the set of negatively recurrent points coincide with the set of positively recurrent points up to a set with zero volume.
\begin{proposition}\label{prop:C_+} The sets $\cC_+$ and $\cC_-$ coincide up to a set with zero volume. \\
Moreover, for every embedded leaf $\Sigma$
    included in $\{ H \neq 0\}$ with a proper inclusion
    map,   up to a set with zero  $(n-1)$-dimensional volume,  
    $\cC_+\cap\Sigma$ coincides $\cC_-\cap\Sigma$.

\end{proposition}
\begin{proof}
We will prove that up to a set with zero volume $\cC_+\subset \cC_-$ and we will deduce the first part of Proposition \ref{prop:C_+}.
We keep the notation of the proof of Proposition \ref{prop:CD}. The first return map  $f_k:\cC'_k\righttoleftarrow$ preserves the finite volume $\frac{1}{H^n}\omega^n$, hence almost every point of $\cC'_k$ is negatively recurrent for $f_k$. This implies that up to a set with zero volume, $\cC'_+$ and hence $\cC_+$ is a subset of $\cC_-$.\\
The proof of the last part is similar. 
\end{proof}
\subsection{An example where $\cC_+\neq \cC_-$}
Adapting the construction made in Section \ref{se:oscillating} and the shadowing lemma, it is not hard to obtain an orbit that is negatively dissipative and positively conservative, i.e.  such that $\cC_-\neq\cC_+$.

\subsection{Boundedness of $r_t$ and invariant measures}

Let us assume that $L\subset M$ is an invariant measurable set of
the dynamics on which $(t,x)\mapsto r_t(x)$ is
a bounded map $\R\times L\to \R$
(in particular $L\subset  \cC_+\cap\cC_-$).
Inspired by the proof of \cite[Theorem~5.1.13]{KatokHasselblatt},
let us define the bounded measurable map $h:L\to\R$,
\begin{equation}\label{eq:h}
    h(x) := \sup_{t\in\R} r_t(x).
\end{equation}
Then $h\circ\varphi_t = h - r_t$, so that for instance Lemma~\ref{lem:rt} implies
that, restricted to $L$,
\begin{equation*}
    \varphi_t^*(e^h\omega)=e^h\omega,\ 
    \varphi_t^*(e^h\ud_\eta H)=e^h\ud_\eta H
    \text{ and }
    (e^hH)\circ\varphi_t = e^h H,\quad
    \forall t\in\R.
\end{equation*}

\begin{corollary}\label{cor:rtBoundedMeasure}
    An invariant measurable set $L\subset M$ of positive measure on which
    $(t,x)\mapsto r_t(x)$ is bounded admits an invariant Borel measure
    of positive density.
\end{corollary}

\begin{proof}
    The measure $A\mapsto \int_A e^{nh}\omega^n$ is positive and invariant.
\end{proof}

\begin{corollary}\label{cor:rtBoundedLeafMeasure}
    If $L$ is an embedded leaf of $\cF$
    on which $(t,x)\mapsto r_t(x)$ is bounded,
    then it admits an invariant Borel measure of 
    positive
    density.
\end{corollary}

\begin{proof}
    The desired measure is $A\mapsto \int_A e^{(n-1)h}\mu$,
    where $\mu$ is given by Proposition~\ref{prop:leafvolume}.
\end{proof}

 We will see in Section \ref{ssexactH=0} that when
$\eta$ is exact in the neighborhood of
$\{ H=0\}$, the flow is conservative and Corollaries~\ref{cor:rtBoundedMeasure}
and \ref{cor:rtBoundedLeafMeasure}
apply.

One of the dynamical importance of these corollaries is signified by
Poincaré's recurrence theorem: if the invariant measures in question
are also finite,
almost every point of the invariant set is recurrent.
However, with the exception of regular leafs of $\cF$
included in $\{ H=0\}$, the recurrence can also be
deduced from Proposition~\ref{prop:CD}.

\subsection{An oscillating behavior}\label{se:oscillating}
In this subsection, we give an example of a flow possessing
orbits included in $\cC_+$   that are in $\{H\neq 0\}$, positively recurrent and whose $\omega$-limit set intersects $\{ H=0\}$. Hence they have an unbounded associated winding $t\in[0, +\infty)\mapsto r_t(x)$.

\begin{proposition}\label{prop:oscillation}
There exists a Hamiltonian map $H:M\to\R$ on some conformal symplectic
manifold, the flow of which satisfying
\begin{equation*}
    \limsup_{t\to+\infty} r_t(x) > \liminf_{t\to+\infty} r_t(x) = -\infty,
\end{equation*}
for some point $x\in\{ H\neq 0\}$.
\end{proposition}

\label{shadowing}
In order to prove this proposition,
let us briefly recall the statement of the Shadowing Lemma for flows
(see \emph{e.g.} \cite[Theorem~18.1.6]{KatokHasselblatt}).
Let $(\varphi_t)$ be a smooth flow on a Riemannian manifold $M$,
the infinitesimal generator of which is $X_t$.
A differentiable curve $c:I\to\R$, $I\subset\R$ interval,
is called an $\e$-orbit if $\|\dot{c}(t)-X_t(c(t))\|\leq\e$ for all
$t\in I$.
A differentiable curve $c:I\to\R$ is said to be $\delta$-shadowed
by the orbit $(\varphi_t(x))_{t\in J}$ if there exists $s:J\to I$
with $|s'-1|<\delta$ such that $d(c(s(t)),\varphi_t(x))<\delta$
for all $t\in J$ ($d$ denoting the Riemannian distance).
The Shadowing Lemma states that, given a hyperbolic set $\Lambda$ of
$(\varphi_t)$, there is a neighborhood $U\supset\Lambda$,
so that for every $\delta>0$ there is an $\e>0$ such that every
$\e$-orbit included in $U$ is $\delta$-shadowed by an orbit
of $(\varphi_t)$.
\begin{proof}
Let $\Sigma$ be a closed hyperbolic surface,
let us denote $\pi:T^1\Sigma\to\Sigma$ the associated unit
tangent bundle and
let $\beta$ be a non-exact closed $1$-form of $\Sigma$.
Let us denote $(G_t)$ the geodesic flow on $T^1\Sigma$
and $X$ the associated vector field.
Let $(M,\eta,\omega)$ be a conformal symplectic closed manifold
associated with $(T^1\Sigma,\pi^*\beta)$ by
Lemma~\ref{lem:lagrangianEmb} in Appendix~\ref{se:isotropic}: that is
one may assume that $L:=T^1\Sigma\times S^1$  is
a Lagrangian submanifold of $M$ and that
the restriction of $\eta$ to this submanifold
is $\alpha:=\pi^*\beta - \ud\theta$ (we identify $\pi^*\beta$
with its pull-back by the projection by a slight
abuse of notation).
Let $W$ be a Weinstein neighborhood of $L$:
identifying the $0$-section of $T^*_\alpha L$ with
$L$, one can see $W$ as a neighborhood of the
$0$-section of $T^*_\alpha L$ (see
\cite[Theorem~2.11]{ChantraineMurphy2019} or
Section~\ref{se:weinstein}).
Let us identify the vector field $X$ of $T^1\Sigma$
with the vector field $X\oplus 0$ of $L$ and
let $H:M\to\R$ be a Hamiltonian function satisfying
$H(q,p)=p(X(q))$ on $W\subset T^*_\alpha L$
(shrinking $W$ if necessary).
Let us prove that $H$ satisfies the statement of
the proposition.

Let us first find an orbit $(\gamma,\dot{\gamma}):
\R_+\to T^1\Sigma$ of the geodesic flow $(G_t)$ such that
\begin{equation}\label{eq:oscillation}
\begin{cases}
\exists K>0,\forall t>0,\quad
\int_{\gamma|[0,t]} \beta \leq K,\\
\limsup_{t\to+\infty} \int_{\gamma|[0,t]} \beta
> \liminf_{t\to+\infty} \int_{\gamma|[0,t]} \beta
= - \infty.
\end{cases}
\end{equation}
Such an orbit can be found applying the Shadowing Lemma
to $(G_t)$.
Indeed, let us fix $\delta\in (0,1)$ and take
an $\e>0$ associated by the Shadowing Lemma.
Let $a:\R/T\Z\to\Sigma$ be a closed geodesic
of unit speed such that $\int_a\beta > 0$
(up to reparametrization,
such an $a$ can be obtained as a minimizer of the energy
functional among loops homotopic to a loop $b$
satisfying $\int_b\beta >0$).
By topological transitivity of $(G_t)$,
there exists an $\e/2$-orbit $(c,\dot{c}):[0,T']\to T^1\Sigma$
such that $\dot{c}(0)=\dot{a}(0)$ and
$\dot{c}(T') = -\dot{a}(0)$.
By successively concatenated $c$ or $c^{-1}$
with higher and higher iterations of $a$ and $a^{-1}$,
one can thus build an $\e$-orbit $(\tilde{\gamma},
\dot{\tilde{\gamma}}):\R_+\to T^1\Sigma$
satisfying conditions (\ref{eq:oscillation})
where $\gamma$ is replaced with $\tilde{\gamma}$.
The Shadowing Lemma applied to this $\e$-orbit
gives us the desired $\gamma$.

According to Section~\ref{se:cotangent},
on $W\subset T^*_\alpha L$, the Hamiltonian flow
$(\varphi_t)$ of $H$ takes the form
\begin{equation*}
    \varphi_t(q,p;z) =
    (G_t(q),e^{r_t(q)} p\circ (\ud G_t(q))^{-1} ;z),
\end{equation*}
where $(q,p)\in T^*(T^1\Sigma)$, $z\in T^* S^1$ and
\begin{equation*}
    r_t(q) := \int_0^t \pi^*\beta(\partial_s G_s(q))
    \ud s,
\end{equation*}
as long as $\varphi_s(q,p;z)$ stays inside $W$
for $s\in [0,t]$.
As $(G_t)$ is Anosov, one has the bundle decomposition
$T(T^1\Sigma) = E^s\oplus \R X\oplus E^u$ which
is preserved by $(G_t)$ with
$\ud G_t\cdot X = X\circ G_t$.
  Let $q\mapsto P_q$ be the section of
$T^*(T^1\Sigma)$ vanishing on $E^s\oplus E^u$ and such that $P(X)\equiv 1$;
it satisfies $P_q\circ (\ud G_t(q))^{-1} = P_{G_t(q)}$ for all $q$.
For fixed $z\in T^* S^1$ and $\lambda >0$, let us
consider the $\R_+$-orbit generated by
$(\dot{\gamma}(0),\lambda P_{\dot{\gamma}(0)};z)$
(where $\gamma$ satisfies (\ref{eq:oscillation})).
By the first condition of (\ref{eq:oscillation}), $r_t(\dot{\gamma}(0))$ is
bounded from above so this orbit keeps inside $W$
for a sufficiently small $\lambda$.
The second condition of (\ref{eq:oscillation}) implies
the statement for $x=(\dot{\gamma}(0),\lambda P_{\dot{\gamma}(0)};z)$
(the orbit is in $\{ H\neq 0\}$
since $P(X)\equiv 1$).
\end{proof}

\section{Global conservative behaviors}\label{se:etaExact}

As we have seen in Corollary~\ref{cor:attractorLee},
a necessary condition for attractors to appear is
the non-exactness of the Lee form in the neighborhood
of $\{ H=0\}$.
Here, we study the opposite case:
Hamiltonian flow $(\varphi_t)$ of $H$ on a closed conformal symplectic
manifold $(M^{2n},\eta,\omega)$ in the case where $\eta$ is exact
in the neighborhood of the invariant set $\{ H = 0\}$.
That is, we assume that there exists an open set $U$
containing $\{ H=0\}$ such that
$[\eta|_U]=0$ in $H^1(U;R)$.
This hypotheses is thus gauge invariant.

\subsection{When $H$ does not vanish}\label{ssHnonnul}
Let us first assume that $H$ does not vanish and denote $X_H^\eta$
its associated vector field for the Lee form $\eta$.
Possibly reversing time, we will assume that $H$ is positive.
Since $X_H^\eta = X_{e^fH}^{\eta+\ud f}$,
by setting $f=-\log\circ H$, we see that $X_H^\eta$ is
the Lee vector field of $\eta'=\eta - \ud(\ln\circ H)$.
Therefore, Hamiltonian flows of non-vanishing $H$ are Lee flows.

We now assume that $H\equiv1$ for the choice of gauge $(\eta,\omega)$,
so that the vector field is $L^\eta$.
Since $\eta(L^\eta)=0$, $r_t\equiv 0$   and $\cC_+=\cC_-=M$, i.e. the flow is positively and negatively conservative with the terminology given in the introduction. Thus almost every point is positively and negatively recurrent.

Lemma~\ref{lem:rt} implies that $\omega$ is preserved by the flow.
Let us point out that this flow is not conjugated to
a symplectic flow in general since one can
have $H^2(M;\R)=0$ (\emph{e.g.} the conformal symplectization
of the contact sphere $(\ES^{2n-1},\frac{1}{2}(x\ud y-y\ud x))$).
The volume form $\omega^n$ is preserved so almost every point
is recurrent according to Poincaré's recurrence theorem.
More precisely, almost every point of a proper
embedded leaf of $\cF$ is recurrent 
according to Proposition~\ref{prop:CD} and
Corollary~\ref{cor:rtBoundedLeafMeasure}.
Let us remark that in the case where
$\eta$ is   completely resonant,    (\emph{i.e.} the subgroup $\{\int_\gamma \eta; \gamma:S^1\to M\}$ is discrete),
there exists $k\in\R^*$ and a map
$\theta : M\to \R/k\Z$ such that $\eta = \ud\theta$
and the invariant foliation $(\theta^{-1}(s))_{s\in\R/k\Z}$ has compact leafs.

\subsection{When
$\eta$ is exact in the neighborhood of
$\{ H=0\}$}\label{ssexactH=0}
Let us move on to the general case:
there exists an open neighborhood $U$ of
$\{ H=0\}$ on which $\eta|_U=\ud\theta$
for some $\theta:U\to\R$.
\begin{proposition}\label{prop:rtborne}
Under the hypotheses of this section,
the map $(t,x)\mapsto r_t(x)$ is bounded
on $\R\times M$.
\end{proposition}
\begin{proof}
Let $\varepsilon_0>0$ be such that
the neighborhood $V_0 := H^{-1}([-\varepsilon_0,
\varepsilon_0])$ is included in $U$ and
let $V := H^{-1}([-\varepsilon,\varepsilon])$ for
some $\varepsilon\in (0,\varepsilon_0)$.
Let $A := \max_{V_0} \theta - \min_{V_0} \theta$,
$b := \inf_{M\setminus V} |H|$ and
$B := \sup_{M\setminus V} |H|$.
We will show that
\begin{equation*}
    |r_t(x)| \leq 2A + \log(B/b),\quad
    \forall (t,x)\in \R\times M.
\end{equation*}
Let $(t,x)\in\R\times M$,
if $\varphi_s(x)\in V_0$ for all $s\in [0,t]$
($[t,0]$ if $t<0$),
then $|r_t(x)|\leq A$.
If $x\in M\setminus V$ and $\varphi_t(x)\in M\setminus V$,
then $b/B\leq |H(\varphi_t(x))/H(x)|\leq B/b$
so $|r_t(x)|\leq \log(B/b)$ according to
Lemma~\ref{lem:rt}.
If $x\in V$ and $\varphi_t(x)\not\in V$,
we assume $t>0$,
let $t_0 \in [0,t]$ such that $\varphi_s(x)\in V_0$
for all $s\in [0,t_0]$ and $\varphi_{t_0}(x)\in M\setminus V$,
then $|r_{t_0}(x)|\leq A$ whereas
$|r_{t-t_0}(\varphi_{t_0}(x))|\leq \log(B/b)$
by the above case, implying $|r_t(x)| \leq A + \log(B/b)$.
The same is symmetrically true for $t<0$ and with $x$ and $\varphi_t(x)$
intertwined.
The last case is when $x\in V$ and $\varphi_t(x)\in V$,
we assume $t>0$ (the other case is symmetrical),
and $\varphi_{s_0}(x)\not\in V_0$ for some $s_0\in [0,t]$.
One can find $t_1<s_0<t_2$ such that
$H(\varphi_{t_1}(x))=H(\varphi_{t_2}(x))=\varepsilon$ and
$\varphi_s(x)\in V$ for $s\in [0,t_1]\cup[t_2,t]$.
By Lemma~\ref{lem:rt}, $r_{t_2}(x)=r_{t_1}(x)$ and
the conclusion follows from the first case treated. 
\end{proof}

Therefore, according to the conservative-dissipative
decomposition of Section~\ref{se:CD},
$M=\cC_-=\cC_+$ and almost every points of $M$ or an embedded
leaf of $\cF$ in $\{ H \neq 0\}$ is   positively and negatively recurrent.
Moreover,
according to Corollaries~\ref{cor:rtBoundedMeasure}
and \ref{cor:rtBoundedLeafMeasure},
$M$ and every embedded leaf of $\cF$   in $\{ H\neq 0\}$ admits an invariant Borel measure of positive density.
In particular, almost every point of a closed regular leaf
of $\cF$ in $\{ H= 0\}$ is also   positively and negatively recurrent.

\subsection{A topologically transitive Lee flow}\label{sstransLee}

Our goal is to provide examples of topologically transitive
Lee flow in every dimension.

  We have seen in Section \ref{ssaxaconsdim2} that  in dimension 2, there are very simple examples of minimal Lee flow.   We recall it. 
Let $\T^2=\R^2/\Z^2$ denote the $2$-torus with canonical coordinates
$x,y\in \R/\Z$.
Let us fix $a,b\in\R$ and
endow $\T^2$ with the conformal symplectic structure
$(\eta,\omega) = (a\ud x + b\ud y,\ud x\wedge\ud y)$,
the Lee flow of which is $\varphi_t(x,y) = (x + bt,y-at)$.
This flow is minimal if and only if $a$ and $b$ are rationally independent.

One way to extend this example is to remark that in the case $a=-1$,
it corresponds to the $\beta$-twisted conformal symplectization
of the contact manifold $(S^1,\ud y)$ with $\beta = b\ud y$.
\begin{proposition}\label{prop:transitiveLee}
    Let $(Y,\alpha)$ be a closed connected contact manifold with a fixed
    contact form, the Reeb flow of which is Anosov and possess a
    periodic orbit $\gamma$ such that $\int_\gamma \beta$ is irrational
    for some closed $1$-form $\beta$.
    Then,
    the standard Lee flow of $\Sconf_\beta(Y,\alpha)$ is topologically
    transitive.
\end{proposition}

Such a $(Y,\alpha)$ can be found in every dimension.
Indeed, let $N$ be a closed Riemannian manifold with negative sectional curvature
and a non-trivial real homology group of degree 1: $H^1(N;\R)\neq 0$.
Let thus $\beta'$ be a non-exact closed 1-form such that
$\int_c \beta'$ is irrational for some loop $c$.
Let $\pi:T^1 N\to N$ be the unit sphere bundle of $N$ endowed
with its usual contact structure (the Reeb flow of which is
the geodesic flow), the $\beta$-twisted conformal
symplectization of $Y:=T^1 N$ with $\beta := \pi^*\beta'$ follows
the hypothesis of the statement.
Indeed, by taking a minimum of the energy functional among loops
homotopic to $c$, one gets a closed geodesic homotopic to $c$,
so a periodic orbit $\gamma$ of the geodesic flow such that
$\int_\gamma \beta = \int_c \beta'$.

However, Lee flows induced by these choices of $(Y,\alpha)$
have a lot of periodic orbits in dimension $2n\geq 4$ so are not minimal.
\begin{proof}[Proof of Proposition~\ref{prop:transitiveLee}]
    Let $(\varphi_t)$ be the Reeb flow of $Y$.
    According to Lemma~\ref{lem:HamiltonSconf},
    the Lee flow $(\Phi_t)$ of $\Sconf_\beta(Y,\alpha) = Y\times S^1$
    takes the following form: for all $(x,\theta)\in Y\times S^1$,
    $t\in\R$,
    \begin{equation*}
        \Phi_t(x,\theta) = (\varphi_t(x),\theta+\rho_t(x)),\text{ where }
        \rho_t(x) := \int_0^t \beta(\partial_s\varphi_s(x))\ud s.
    \end{equation*}
    In order to show topological transitivity, it is enough to
    prove that for every pair of product     non-empty open sets
    $U_i\times V_i\subset Y\times S^1$, $i\in\{ 1,2\}$,
    there is some $(x,\theta)\in U_1\times V_1$ and
    some $t\in\R$ such that $\Phi_t(x,\theta)\in U_2\times V_2$
    (see \emph{e.g.} \cite[Lemma~1.4.2]{KatokHasselblatt}).
     We can assume that the $V_i$ are arcs of length $\ell>0$.\\
    By hypothesis, there exists
    a point $y\in Y$ such that $\varphi_{t_2}(y)=y$
    for some $t_2>0$ and $\rho_{t_2}(y)$ is irrational.\\
     We choose $x_i\in U_i$ for $i=1, 2$.
    Let $\delta>0$ be so small that if a curve $c:[a,b]\to Y$
    with $c(a)=x_1$ and $c(b) =  x_2$ is $\delta$-shadowed
    by an orbit $\nu:[a',b']\to Y$, then $\nu(a')\in U_1$, $\nu(b')\in U_2$
    and $|\int_c\beta - \int_\nu \beta| < \ell/3$.

    By assumption, $(\varphi_t)$ is a topologically transitive
    Anosov flow (contact Anosov flows on connected manifolds are
    topologically mixing \cite[Theorem~18.3.6]{KatokHasselblatt}).
    
    Let $\e>0$ be associated with $\delta$
    by the Shadowing Lemma.
    By transitivity, there exist $\e/2$-orbits $c_1:[0,t_1]\to Y$,
    $c_3:[0,t_3]\to Y$,
      with $c_1(0) = x_1$, $c_1(t_1)=y$ and $c_3(0) = y$, $c_3(t_3)=x_2$.
    Then (up to a small deformation at the connecting points)
    the concatenated paths $\nu_k:=c_1\cdot c_2^k \cdot c_3$, $k\in\N$,
    are $\e$-orbits for $(\varphi_t)$.
    Let $R_k := \int_{\nu_k}\beta \bmod 1 \in S^1$.
    Since $\int_{c_2}\beta$ is irrational, $(R_k)$ is dense in $S^1$.
    Let us fix $k\in\N$ such that $\theta + [R_k - \ell/3,R_k + \ell/3]
    \in V_2$ for some $\theta\in V_1$ (the length of the arcs $V_i$'s being $>\ell$).
    Applying the Shadowing Lemma to $\nu_k$, we find an orbit
    $\nu:[0,T]\to Y$ such that $\nu(0)\in U_1$, $\nu(T)\in U_2$
    and $\int_\nu \beta\bmod 1$ is $\ell/3$-close to $R_k$
    so that $\theta+ \int_\nu \beta\in V_2$.
\end{proof}
    
\subsection{A Lee flow with no periodic orbit}\label{ssLeenoperiodic}
Relaxing the transitivity hypothesis, one can easily produce
a Lee flow without periodic orbit.

\begin{proposition}\label{prop:aperiodicLee}
    Let $\T^n$ be the flat $n$-torus with canonical coordinates $x_i\in\R/\Z$,
    and let $\beta := a_1\ud x_1 +\cdots + a_n \ud x_n$ for some fixed
    $a_i\in\R$.
    The standard Lee flow of $\Sconf_\beta(T^1\T^n)$ does not
    have any periodic orbit if and only if the family
    $(1,a_1,\ldots,a_n)$ is
    rationally independent.
\end{proposition}

This flow is clearly not minimal and, in general, there is not much
hope for the standard Lee flow of a closed twisted conformal symplectization
to be minimal in dimension $2n\geq 4$.
Indeed, the Weinstein conjecture states that every Reeb flow
of a closed contact manifold $(Y,\alpha)$ should possess a
closed orbit $\gamma$, so $\gamma\times S^1$ would be
a closed invariant set of the standard Lee flow
of the twisted conformal symplectizations
of $(Y,\alpha)$.

\begin{proof}[Proof of Proposition~\ref{prop:aperiodicLee}]
The Reeb flow of $T^1\T^n\simeq\T^n\times\ES^{n-1}$
is $\varphi_t(x,v):=(x+tv,v)$.
The associated $\rho_t(x,v) := \int_0^t
\beta(\partial_s\varphi_s(x,v))\ud s\bmod 1$
satisfies $\rho_t(x,v) = \sum_i a_i tv_i\bmod 1$.
Therefore, a point $(x,v,\theta)\in T^1\T^n\times S^1$
is a $\tau$-periodic point of the Lee flow
if and only if
$\tau v\in \Z^n$ and $\sum_i a_i \tau v_i\in\Z$.
\end{proof}
\subsection{  A conservative behavior with $\eta_{|\{ H=0\}}$ non exact}\label{ssconsnonexact}
 \begin{proposition}
    Let $\T^n$ be the flat $n$-torus with canonical coordinates $q_i\in\R/\Z$,
    and let us consider on $\Sconf(T^1\T^n)$ the Hamiltonian $H( q_1, \dots,q_n, p_1, \dots, p_n, \theta)=p_1 $. The Hamiltonian flow is
    $$\varphi_t^H( q_1, \dots,q_n, p_1, \dots, p_n, \theta)=(q_1+t,q_2, \dots,q_n, p_1, \dots, p_n, \theta).$$
This flow preserves the conformal 2-form and the  zero level $\{ H=0\}$ contains a loop $\gamma$ such that 
$$\int_\gamma\eta\neq 0.$$
\end{proposition}
 \begin{proof}
The contact form is the restriction of the Liouville 1-form $\lambda$ to $T^1\T^n$ and the Reeb   vector field $R$ at $(q, p)$ is $(p, 0)$. Hence $dH\cdot R=0$ and 
     the contact Hamiltonian flow $X_H$ satisfies  $\iota_{X_H}d\lambda=-dp_1$ and $X_H=(1, 0, \dots, 0)$.
     As $dH\cdot R=0$, we deduce from Lemma \ref{lem:HamiltonSconf} that the conformal Hamiltonian vector field is $( 1,0,  \dots, 0)$.
\end{proof}

\section{Dissipative behaviors}\label{Sdiss}

\subsection{Dissipative ergodic measures}  

Let $\nu$ be an ergodic measure and let us
denote
\begin{equation*}\label{eq:meanr}
\meanr(\nu) := \int_M \eta(X)\ud\nu.
\end{equation*}
The following proposition is the ergodic
counterpart of Corollary~\ref{cor:01law}.

\begin{proposition}\label{eq:ergodic01law}
    Given an ergodic measure $\nu$,
    $\nu(\{ H=0\})\in\{ 0,1\}$ and in the case
    where $\meanr(\nu)\neq 0$, the support of $\nu$
    is included in $\{ H=0\}$.
\end{proposition}

\begin{proof}
    The first statement is obvious since $\{ H=0\}$
    is an invariant set.
    If $\meanr(\nu)\neq 0$ and
    $\supp\nu\not\subset \{H=0\}$,
    there exists $x\in\{ H\neq 0\}$
    such that $r_t(x)\sim t\meanr(\nu)$ as
    $t\to\pm\infty$.
    However $H$ is bounded and $H\circ\varphi_t = e^{r_t} H$
    (Lemma~\ref{lem:rt}), a contradiction.
\end{proof}

Let us recall the result of Liverani-Wojtkowski
about the Lyapunov spectrum of conformally
symplectic cocycles \cite{LiveraniWojtkowski1998}.
We  state the results in the invertible case.
Let $(M,\nu)$ be a probability space with
an inversible ergodic map $T:M\to M$ and let
$A:M\to GL(\R^{2n})$ be a measurable map
such that both measurable maps $\log_+\| A^{\pm 1}\|$ are integrable
defining the cocycle
$A^m(x) := A(T^{m-1}x)\cdots A(x)$ for $m\in\Z$.
According to Oseledets multiplicative ergodic theorem,
there exists real numbers $\lambda_1<\cdots <\lambda_s$
called the Lyapunov exponents of $A$ and
an associated decomposition of $\R^{2n}$
(that we will call the Lyapunov decomposition of $\R^{2n}$)
in linear subspaces $F_1(x)\oplus\cdots\oplus F_s(x)$
defined for $\nu$-almost every $x\in M$, such that
\begin{equation*}
    \lim_{m\to\pm\infty} \frac{1}{m}
    \log\| A^m(x)v\| = \lambda_k,\quad
    \forall v\in F_k(x),\forall k\in\{ 1,\ldots, s\}.
\end{equation*}
The positive integer $d_k := \dim F_k$
is well-defined and called the multiplicity of $\lambda_k$,
these multiplicities satisfy
\begin{equation*}
    \sum_{k=1}^s d_k \lambda_k = \int_M
    \log |\det A|\ud \nu.
\end{equation*}
Liverani-Wojtkowski showed a symmetry of the Lyapunov
spectrum in the case where $A$ takes its values
in the conformally symplectic linear group
$CSp(2n):=CSp(\R^{2n},\omega_0)$.
A conformally symplectic linear map $S\in CSp(E,\omega)$
is a linear map of a symplectic linear space $(E,\omega)$
satisfying $S^*\omega = \beta\omega$
for some $\beta\in\R^*$ called the conformal factor of $S$.

\begin{theorem}[{\cite[Theorem~1.4]{LiveraniWojtkowski1998}}]\label{thm:Liverani}
    Let $(M,\nu)$ be a probability space with an invertible
    ergodic map $T:M\to M$ and let $A:M\to CSp(2n)$
    be a measurable cocycle such that
    $\log_+ \| A^{\pm 1}\|$ are integrable.
    Let $\beta:M\to\R^*$ be such that
    $A(x)^*\omega_0 = \beta(x) \omega_0$ for all $x\in M$.
    Then we have the following symmetry of the Lyapunov
    spectrum $\lambda_1<\cdots<\lambda_s$ of $A$:
    \begin{equation*}
        \lambda_k + \lambda_{s-k+1} = b,
        \text{ where }
        b := \int_M \log |\beta|\ud\nu,
    \end{equation*}
    for every $k\in\{ 1,\ldots, s\}$.
    Moreover, the subspace $F_1\oplus\cdots\oplus F_{s-k}$
    is the $\omega_0$-orthogonal complement of
    $F_1\oplus\cdots\oplus F_k$.
\end{theorem}

Let us come back to our setting and consider an
ergodic measure $\nu$ of $M$ for the Hamiltonian
flow $(\varphi_t)$.
Let us fix a Riemannian metric $g$ on $M$.
By taking a measurable
symplectic trivialization of $TM$,
one naturally extends the Oseledets multiplicative
ergodic theorem for measurable maps
$A:M\to GL(\R^{2n})$ such that $\log_+\|A^{\pm 1}\|$ are
integrable for $\nu$ to measurable section
$A:M\to GL(TM)$ of the fiber bundle $GL(TM)$
such that $\log_+\|A^{\pm 1}\|$ are integrable
where $\|\cdot\|$ is the Riemannian operator norm
associated with $g$.
The section $A:x\mapsto \ud\varphi_1(x)$ satisfies
the integrability condition and the
cocycle $A^m$ corresponds to $\ud\varphi_m$ for $m\in\Z$.
The associated Lyapunov exponents
$\lambda_1<\cdots<\lambda_s$ define the
Lyapunov exponents of the flow $(\varphi_t)$
for the ergodic measure $\nu$.
By compactness of $M$, $t\mapsto \partial_t(\log \|\ud\varphi_t(x)v\|)$ is bounded,
$x\in M$ and $v\in T_x M\setminus \{0\}$ being fixed,
so 
\begin{equation*}
    \lim_{m\to\pm\infty} \frac{1}{m}\log\|\ud\varphi_m(x)v\|
    = \lim_{t\to\pm\infty} \frac{1}{t}\log\|\ud\varphi_t(x)v\|,
\end{equation*}
for every $(x,v)\in TM$ for which one of the limit is defined,
where $m\in\Z^*$ and $t\in\R^*$.

\begin{corollary}\label{cor:appliedLiverani}
    Let $\nu$ be an ergodic measure of the Hamiltonian flow
    $(\varphi_t)$ of the closed conformal symplectic manifold $M$.
    Let $\lambda_1<\cdots <\lambda_s$ be the associated Lyapunov spectrum
    and $F_1,\ldots, F_s$ be the associated Lyapunov decomposition of $TM$.
    For every $k\in\{ 1,\ldots, s\}$,
    \begin{equation*}
        \lambda_k + \lambda_{s-k+1} = \bar{r}(\nu)
    \end{equation*}
    and the subbundle $F_1\oplus\cdots\oplus F_{s-k}$ is the
    $\omega$-orthogonal complement of $F_1\oplus\cdots\oplus F_k$.
\end{corollary}

\begin{proof}
We apply Theorem~\ref{thm:Liverani} to the section
$x\mapsto \ud\varphi_1(x)$ of
the subbundle of conformally symplectic linear
maps of $TM$, with associated conformal factor
$\beta:x\mapsto e^{r_1(x)}$.
We only need to prove that $b:=\int_M \log|\beta|\ud\nu$
equals $\bar{r}(\nu)$.
By Fubini's theorem and
invariance of $\nu$,
\begin{equation*}
    b = \int_M r_1 \ud\nu =
    \int_M \int_0^1 \eta(X\circ\varphi_t(x))\ud t\ud\nu (x)
    = \int_0^1 \int_M \eta(X)\ud\nu \ud t = \bar{r}(\nu).
\end{equation*}
\end{proof}

Let us remark that the fact that $F_1\oplus\cdots\oplus F_{s-k}$
is the $\omega$-orthogonal complement of
$F_1\oplus\cdots\oplus F_k$ for every $k$ implies that
\begin{equation}\label{eq:dualFnonorthogonal}
    F_k^\omega \cap F_{s-k+1} = 0,\quad \forall k\in \{ 1,\ldots, s\},
\end{equation}
$\nu$-almost everywhere.

\begin{corollary}
    Let $\nu$ be an ergodic measure of a Hamiltonian flow $(\varphi_t)$
    of the closed conformal symplectic manifold $M$.
    There is a measurable sub-bundle $F$ of the
    Lyapunov decomposition of $TM$ which
    is transverse to $\cF$
    on which, for $\nu$-almost every $x\in M$,
    \begin{equation*}
        \lim_{t\to\pm\infty } \frac{1}{t}\log \|\ud\varphi_t(x)v\| =
        \bar{r}(\nu),\quad \forall v\in F(x)\setminus \{0\}.
    \end{equation*}
\end{corollary}

\begin{proof}
    Let $\lambda_1 <\cdots < \lambda_s$ be the associated Lyapunov spectrum and
    $F_1,\ldots, F_s$ be the associated decomposition of $TM$.
    Let $X$ be the vector field of $(\varphi_t)$.
    Since $\R X$ is invariant with $\ud\varphi\cdot X = X\circ\varphi$,
    there is $k\in\{ 1,\ldots,s\}$ such that $\lambda_k = 0$
    and $\R X\subset F_k$ $\nu$-almost everywhere.
    Let us show that $F:=F_{s-k+1}$ is the desired sub-bundle.
    According to Corollary~\ref{cor:appliedLiverani},
    $\lambda_{s-k+1} = \bar{r}(\nu)$.
    According to (\ref{eq:dualFnonorthogonal}),
    $\omega(X,v)\neq 0$ for some $v\in F\setminus \{0\}$
    when $X\neq 0$. Since $\ud_\eta H=\iota_X\omega$, the conclusion
    follows.
\end{proof}

Let $r\in\{ 1,\ldots, s\}$ be the maximal integer such that $\lambda_r<0$.
According to the non-linear ergodic theorem of Ruelle 
\cite[Theorem~6.3]{Ruelle1979}, for every $k\in \{ 1,\ldots, r\}$ and
for $\nu$-almost every $x\in M$,
the set
\begin{equation*}
    V_k(x) := \left\{ y\in M\ |\ 
    \limsup_{t\to+\infty} \frac{1}{t} \log d(\varphi_t(x),\varphi_t(y))
    \leq \lambda_k \right\},
\end{equation*}
where $d$ denotes the Riemannian distance,
is the image of $F_1(x)\oplus\cdots\oplus F_k(x)$ by a smooth injective immersion tangent to identity
at $x$.
Therefore the last corollary implies the following proposition.

\begin{corollary}
Let $\nu$ be an ergodic measure of a Hamiltonian flow $(\varphi_t)$
of the closed conformal symplectic manifold $M$ such that
$\bar{r}(\nu) < 0$ and such that $\supp\nu$ is included in a
connected component $\Sigma$ of $\{ H=0\}$ without critical point
of $H$.
For $\nu$-almost every point $x$ of $\Sigma$, there exists an
immersed submanifold $V\subset M$ transverse to $\Sigma$ and containing $x$
such that
\begin{equation*}
    \limsup_{t\to+\infty} d(\varphi_t(x),\varphi_t(y)) \leq \bar{r}(\nu),\quad
    \forall y\in V.
\end{equation*}
\end{corollary}

\subsection{Examples of isotropic attractors}\label{ssExampleattrcators}

\subsubsection{Legendrian attractors}\label{ssLegendreattractors}

Contact Hamiltonian dynamical systems can provide examples of
conformal dynamical systems by taking their lift to the conformal
symplectization (which is closed if the contact manifold is closed).

\begin{proposition}\label{PLegendrianattractors}
Every Legendrian submanifold is a hyperbolic
attractor for some autonomous contact Hamiltonian
flow.
\end{proposition}

\begin{proof}
Let $L$ be a Legendrian submanifold of a contact
manifold. According to the contact Weinstein neighborhood
theorem, one can assume that $L$ is the $0$-section
of $(T^*L\times\R,\ud z-y\ud x)$,
with local coordinates $(x,y)\in T^*L$ and
$z\in\R$
(see \emph{e.g.} \cite[Corollary~2.5.9 and
Example~2.5.11]{Geiges2008}).
Given $H:T^*L\times\R\to\R$, the contact
Hamilton equations (\ref{eq:contactHam}) takes the form
\begin{equation*}
    \begin{cases}
    \dot{x} = -\partial_y H,\\
    \dot{y} = \partial_x H + y\partial_z H,\\
    \dot{z} - y\dot{x} = H.
    \end{cases}
\end{equation*}
Choosing $H(x,y,z) = -z$, the flow
is $\varphi_t(x,y,z) = (x,e^{-t}y,e^{-t}z)$.
\end{proof}

Let us give some explicit global examples.
Let us first consider the standard contact sphere
$(\ES^{2n-1},\frac{1}{2}(x\ud y - y \ud x))$.
Since $(\C^n\setminus 0,\ud x\wedge\ud y)$ is
the symplectization of the standard sphere,
every contact Hamiltonian flow can be obtained in
the following way:
let $H:\C^n\setminus \{0\}\to\R$ be a positively
$2$-homogeneous Hamiltonian, the flow of which
is $(\Phi_t)$, then
\begin{equation*}
\varphi_t(z) := \frac{\Phi_t(z)}{\|\Phi_t(z)\|},\quad
\forall z\in\ES^{2n-1},\forall t\in\R,
\end{equation*}
defines a contact Hamiltonian flow of $\ES^{2n-1}$.
Let $H(x,y) :=\frac{1}{2}( \|x\|^2 - \|y\|^2)$,
so that $\Phi_t(x,y) = (\cosh(t)x+\sinh(t)y,
\sinh(t)x+\cosh(t)y)$.
The associated contact flow has one Legendrian
attractor $L_+ := \{ x = y\}$ and one Legendrian
repeller $L_- :=\{ x = - y \}$, every point outside
of them having its $\alpha$-limit set inside $L_-$
and its $\omega$-limit set inside $L_+$.

Let us now consider a vector field $X$ on some
closed manifold $M$ generating a flow $(f_t)$.
According to Section~\ref{se:cotangent},
this flow extends to a Hamiltonian flow $(\hat{f}_t)$
(identifying the $0$-section with $M$)
on $T^*M$ which is fiberwise homogeneous:
$\hat{f}_t(q,a p)=a\hat{f}_t(q,p)$,
$\forall (q,p)\in T^*M$, $\forall a \in \R$.
Let us endow $M$ with a Riemannian metric,
the flow $(\hat{f}_t)$ induces a Contact Hamiltonian
flow $(\varphi_t)$ on the unit cotangent bundle
$(S^*M,i^*\lambda)$ ($\lambda$ being the Liouville form
and $i:S^*M\hookrightarrow T^*M$ the inclusion)
by
\begin{equation*}
    \varphi_t(q,p) :=
    \frac{\hat{f}_t(q,p)}{\|\hat{f}_t(q,p)\|},
    \quad \forall (q,p)\in S^*M.
\end{equation*}
A hyperbolically attracting (resp. repelling) fixed point $x\in M$
of $(f_t)$ corresponds to a normally hyperbolically
attracting (resp. repelling) Legendrian fiber $S^*_x M$ of $(\hat{f}_t)$.

 Let us remark that in both examples, one can directly
work in the conformal symplectization by taking
the flow induced by $(\Phi_t)$ (resp. $(\hat{f}_t)$)
on the quotient space $(\C^n\setminus \{0\})/(z\sim ez)$
(resp. $(T^*M\setminus \{0\})/((q,p)\sim (q,ep))$),
where $e:=\exp(1)$.

\subsubsection{Hyperbolic   attractive and repulsive closed orbit
in every non-symplectic manifold}

Here, by a non-symplectic manifold,
we mean a conformally symplectic manifold,
the conformal structure of which is not $\sim(0,\omega)$.

\begin{proposition}\label{prop:attractingPeriodicOrbits}
Let $(M,\eta,\omega)$ be a conformally symplectic
manifold and let $\gamma : S^1\hookrightarrow M$
be an embedded loop such that $\int_\gamma \eta < 0$.
There exists a Hamiltonian $H:M\to\R$ admitting
$\gamma$ as a hyperbolic attracting periodic orbit.
\end{proposition}

In particular, every non-symplectic manifold
admits  a conformal Hamiltonian flow that has a hyperbolic   attractive periodic orbit   and a hyperbolic repulsive periodic orbit 
Hamiltonian.

\begin{proof}
Let us first remark that $\gamma$ is included
in an open Lagrangian submanifold.
According to Theorem~\ref{thm:weinstein}
in Appendix~\ref{se:isotropic},
using a cut-off function to define $H$ globally,
one can replace $M$ with the normal bundle $N\to S^1$ of
$\gamma\simeq S^1$, identified with the $0$-section.
As a vector bundle $N$ equals $W\oplus T^*S^1$,
where $W$ is a symplectic vector bundle.
Since the Lagrangian Grassmannian is connected, one can find a lagrangian subbundle $L\subset W$.
Then $L$
is a Lagrangian submanifold of $N$ containing $\gamma$.

One can thus assume that $M=T^*_\beta L$ with
$\gamma$ included in the $0$-section
identified with $L$.
The closed form $\beta$ of $L$ is the pull-back
of the Lee form $\eta$ to $L$, in particular
$r:=\int_\gamma \beta <0$.
Let $X$ be a complete vector field of $L$
inducing a flow $(f_t)$ for which $\gamma$
is a 1-periodic hyperbolic orbit
such that the eigenvalues $\mu$'s of $\ud f_1(\gamma(0))$
satisfy $e^r < \mu < 1$.
Let $(\hat{f}_t)$ be the lifted Hamiltonian flow
of $T_\beta^*L$ properly cut-off outside a neighborhood
of $\gamma$ (see the end of Section~\ref{se:cotangent}).
The differential of $\hat{f}_1$ at $\gamma(0)$
is equivalent to $\ud f_1(\gamma(0))\oplus
e^r(\ud f_1(\gamma(0)))^{-1}$ so
its eigenvalues are in $(0,1)$.
\end{proof}
\subsection{Connected components of $\{H=0\}$ and attraction}
Here we wonder if a connected component of $\{ H=0\}$ can be an attractor.  In Section \ref{ssaxadissdim2}, we gave a 2-dimensional example where a connected component of $\{ H=0\}$ is attractive. This is the only example that we know, and here we give conditions that ensure that such a component cannot be attractive. We will say that a subset $\Sigma$ of $M$ separates locally $M$ in two connected components if in every neighbourhood $V$ of $\Sigma$, there exists an open neighbourhood $U\subset V$ of $\Sigma$ such that $U\backslash \Sigma$ has exactly two connected components. The manifold $M$ being connected, we say that $\Sigma$ separates globally $M$ if $M\backslash \Sigma$ is not connected.

\begin{proposition}\label{pH=0sep}
    Assume that $\Sigma$ is an isolated connected component of $\{ H=0\}$ that
    separates globally and locally $M$ in two connected components. Then $\Sigma$
    cannot be a strong attractor.
\end{proposition}

\begin{proof}
Let us assume that $\Sigma$ is a strong attractor. Then there exists an open neighbourhood $U_0$ of $\Sigma$ such that $U_0\cap\{ H=0\}=\Sigma$.
As $\Sigma$ separates locally $M$ in two connected components, there exists a neighbourhood $U$ of $\Sigma$ such that $U\subset U_0$ has two connected components, $U_-$ and $U_+$. We denote by $\varepsilon_\pm\in \{ -1, 1\}$ the sign of $H_{|U_\pm}$.

We know that $M\backslash  \Sigma$  is not connected, and the boundary of each of its connected components intersects $\Sigma$ and thus contains $U_-$ or $U_+$. This implies that $M\backslash \Sigma$ has exactly two connected components, $M_-$ that contains $U_-$ and $M_+$ that contains $U_+$. We denote by $\varepsilon: M\backslash \Sigma\to \{ -1, 1\}$ the function such that $\varepsilon_{|M_\pm}=\varepsilon_\pm$.

We choose a smooth bump function $\chi: M\to [0, 1]$ such that $\chi$ is equal to 1 in a neighbourhood  of $\Sigma$ and the support of $\chi$ is contained in $U$. Then the Hamiltonian $K: M\to\R$ is defined by $K=\chi H+(1-\chi)\varepsilon$.  We have  $\Sigma=\{ K=0\}$. The Hamiltonian flow of $K$
    coincides with the flow of $H$  in a neighborhood of $\Sigma$ and then
    $\Sigma$ is also a strong attractor for $(\varphi_t^K)$. If $x$ is
    a generic point in the basin
    of attraction of $\Sigma$ for $K$ but not in $\Sigma$, $x$ is wandering.
    Observe that a wandering point is wandering for $(\varphi_t^K)$ and
    $(\varphi_t^{-K})$. We deduce from Corollary \ref{Cwandering} that
    $\lim_{t\to-\infty}K(\varphi_t^K(x))=0$ since $x$ was taken generically. But as $\Sigma=\{K=0\}$ is a strong
    attractor, this is not possible.
\end{proof}

We do not know whether a similar statement is true without the separation
assumption. We obtain the following result when we assume normal hyperbolic
attraction.

\begin{theorem}\label{thm:hyperbolic}
    Let us assume $(M,\eta,\omega)$ has dimension
    $2n\geq 4$ and let $H:M\to\R$ be Hamiltonian.
    Let $\Sigma$ be a closed connected component of $\{ H=0\}$
    without critical point of $H$.
    Then $\Sigma$ cannot be hyperbolically normally
    attracting.
\end{theorem}

\begin{proof}
    Let us assume that such a $\Sigma$ is hyperbolically
    normally attracting and reach a contradiction.
    Let us restrict ourself to a neighborhood of $\Sigma$.
    One can assume that $\Sigma=\{ H=0\}$ and that
    $M=\Sigma\times (-\varepsilon_0,\varepsilon_0)$ 
    for some $\varepsilon_0>0$
    with $H(x,y)=y$ for all
    $(x,y)\in\Sigma\times (-\varepsilon_0,\varepsilon_0)$
    by a change of variables in a tubular neighborhood of $\Sigma$.
    Let $V_\varepsilon := \Sigma\times (-\varepsilon,\varepsilon)$, then
    $\Sigma$ being normally
    hyperbolic means that one can assume that there exists
    $a\in (0,1)$ and $\tau>0$ such that
    \begin{equation}\label{eq:normallyHyperbolic}
        \varphi_\tau(V_\varepsilon) \subset V_{a\varepsilon},\quad
        \forall \varepsilon\in (0,\varepsilon_0).
    \end{equation}
    
    According to Proposition~\ref{prop:leafvolume} applied to
    the leaf $\Sigma$, there exists a volume form
    $\mu$ of $\Sigma$ such that $\varphi_t^*\mu = e^{(n-1)r_t}\mu$
     and which is the pull-back of a form $\mu_0$ of $M$
    such that $\mu_0\wedge \ud y = \omega^n$ in the
    neighborhood of $\Sigma$. 
    Let $\pi:M\to\Sigma$ be the projection on the first factor.
    By decreasing $\varepsilon_0$, one can assume that
    $\pi^*\mu\wedge\ud y$ does not vanish so that there
    exists a non-vanishing map $f:M\to \R_+$ 
     
    such that
    $f\pi^*\mu\wedge\ud y = \omega^n$.
   Since $\mu_0\wedge \ud y = \omega^n$ and
    $\mu_0$ and $\pi^*\mu$ agree on $T\Sigma$,
    $f|_\Sigma \equiv 1$.
    By (\ref{eq:normallyHyperbolic}),
    \begin{equation}\label{eq:omegaNormallyHyperbolic}
        \omega^n(\varphi_\tau(V_\varepsilon)) \leq \omega^n(V_{a\varepsilon}),
        \quad\forall \varepsilon\in(0,\varepsilon_0).
    \end{equation}
    On the one hand,
    \begin{equation}\label{eq:omegaVe}
        \omega^n(V_\varepsilon) = \int_{x\in\Sigma} 
        \left(\int_{-\varepsilon}^\varepsilon f(x,y)\ud
        y\right)\mu_x
        \stackrel{\varepsilon\to 0}{\sim} 2\varepsilon\cdot\mu(\Sigma).
    \end{equation}
    On the other hand,
    \begin{equation}\label{Etransfovol}
        \omega^n(\varphi_\tau(V_\varepsilon)) =
        \int_{V_\varepsilon} e^{nr_\tau}\omega^n
        \stackrel{\varepsilon\to 0}{\sim} 2\varepsilon \int_\Sigma e^{nr_\tau}
        \mu.
    \end{equation}
    Therefore, \eqref{eq:normallyHyperbolic},\eqref{eq:omegaVe}, \eqref{Etransfovol} imply
    \begin{equation}\label{eq:ineqSigma}
        \int_\Sigma e^{nr_\tau}\mu \leq a \mu(\Sigma).
    \end{equation}
    By Hölder's inequality,
    \begin{equation*}
        \int_\Sigma e^{(n-1)r_\tau}\mu \leq
        \left(\int_\Sigma \mathbf{1}^n\mu\right)^{\frac{1}{n}}
        \left(\int_\Sigma e^{nr_\tau}\mu\right)^{\frac{n-1}{n}},
    \end{equation*}
    since $\varphi_\tau^*\mu = e^{(n-1)r_\tau}\mu$, it follows from
    (\ref{eq:ineqSigma}) that
    \begin{equation*}
        \mu(\Sigma) = \mu(\varphi_\tau(\Sigma))
        \leq \mu(\Sigma)^{\frac{1}{n}} a^{\frac{n-1}{n}}
        \mu(\Sigma)^{\frac{n-1}{n}},
    \end{equation*}
    so $a\geq 1$, a contradiction.
\end{proof}
\section{Invariant distribution and submanifolds}\label{Sdistsub}

  We introduced in section \ref{ssinvdistr} the invariant distribution $\cF$.
\subsection{ Holonomy of embedded leafs of $\cF$ }

Let us study the holonomy of a regular leaf $F$ of $\cF$.
By definition, one can find an open neighborhood $U$ of $F$
on which $\cF$ defines a non-singular foliation.
The holonomy of $F$ is well-defined as the holonomy of $F$
in $U$ for this foliation.

Let us recall the definition of the holonomy
$\pi_1(F)\to G$ of a leaf $F$ of a foliation $\cG$ of codimension $p$
on a manifold $N^n$.
We refer to \cite{Haefliger}.
A distinguished map $f:V\to\R^p$ of $\cG$ is a map on a trivialization
neighborhood $V\simeq \R^p\times\R^{n-p}$ that factors by
the projection $\R^p\times\R^{n-p}\to \R^p$.
Given a point $z\in F$, let $G$ be the group of germs of
local homeomorphisms of $\R^p$ fixing $0$
defined up to internal automorphisms (\emph{i.e.} up to conjugacy
by such germs).
Given a loop $\gamma:S^1\to F$ based at $z$ and a germ of distinguished
map $f$ sending $z$ to $0$, there is a unique continuous lift $(f_t)$
of $\gamma$ in the space of germs of distinguished maps
such that $f_0 = f$ and
$f_t$ sends $\gamma(t)$ to $0$.
There exists a unique germ $g:\R^p\to\R^p$ fixing the origin
such that $f_1=g\circ f_0$.
This germ only depends on $f$ and the homotopy class of $\gamma$.
If one takes another germ $f'$ of distinguished map at $z$,
the same procedure will give a germ $g':\R^p\to\R^p$ fixing the origin
that is conjugated to $g$.
Therefore, one defines
the holonomy of $F$ (based at $z$) as the morphism
$\pi_1(F,z) \to G$ sending the class of the loop $\gamma$ to
the class of the germ $g$.
The holonomy group of $F$ is the image of the holonomy.
Up to isomorphisms, these notions do not depend on the base point $z$
(a leaf being path-connected).

\begin{proposition}
    The holonomy of an embedded leaf outside $\{ H= 0\}$ is trivial.
    Let $\Sigma\subset \{ H=0\}$ be a connected component of
    $\{ H=0\}$ without critical point of $H$,
    the holonomy of $\Sigma$ is
    \begin{equation*}
        [\gamma] \mapsto \left[ y\mapsto e^{\int_\gamma \eta} y\right].
    \end{equation*}
    In particular, the holonomy group of $\Sigma$ is isomorphic
    to the subgroup $\langle [\eta], \pi_1(\Sigma)\rangle$ of $\R$.
\end{proposition}

\begin{proof}
    One can prove this proposition by considering the global
    distinguished map $e^{-\theta} H\circ p$
    defined on the universal cover $p:\widetilde{M}\to M$
    with $\ud\theta = p^*\eta$. Let us give a more
    intrinsec proof.

    If $F$ is an embedded leaf outside $\{ H = 0\}$,
    the pull-back of the Lee form $\eta$ to $F$ is exact
    according to Corollary~\ref{cor:01law},
    so $\eta = \ud\theta$ on a tubular neighborhood $U$ of $F$.
    Therefore $\cF$ is trivially fibered by $e^{-\theta}H$
    in $U$ and the holonomy is thus trivial.

    Let $\Sigma$ be a connected component of $\{ H=0\}$ without
    critical point.
    Let $i:\Sigma\hookrightarrow M$ be the inclusion map.
    If $i^*\eta$ is exact, the holonomy is trivial, as above.
    Otherwise, let us fix $z\in\Sigma$ such that
    $(i^*\eta)_z\neq 0$, which implies that
    $\ker\eta_z$ is transversed to $T_z \Sigma$.
    Let $T\subset M$ be an open connected $1$-dimensional manifold
    containing $z$ and tangent to $\ker\eta$.
    By shrinking $T$, one can assume that $H$ induces an isomorphism
    $H|_T:T\to (-\varepsilon,\varepsilon)$ sending $z$ to $0$,
    for some $\varepsilon>0$.
    Let us remark that there exist a distinguished map $f$ in the
    neighborhood of $z$ such that $f|_T = H|_T$.
    Indeed, in the neighborhood of $z$, let $\theta$ be such that $\theta(z)=0$
    and $\ud\theta = \eta$, then $f:=e^{-\theta} H$ is a distinguished
    map. Since $T$ is tangent to $\ker\eta$, $\theta|_T\simeq 0$
    so that $f|_T = H|_T$.

    Let $\gamma:[0,1]\to\Sigma$ be a smooth loop based at $z$.
    According to \cite[\S 2.5]{Haefliger}
    for every $x\in W$ a connected neighborhood of $z$ in $T$,
    there are smooth paths $\gamma_x : [0,1] \to M$ tangent
    to $\cF$ and $C^0$-close to $\gamma$ such that,
    $\gamma_x(0) = x$, $\gamma_x(1)\in T$ and the image of the holonomy
    $\pi_1(\Sigma,z)\to G$ at $[\gamma]$ is the class of the germ
    \begin{equation*}
         y\mapsto H\left(\gamma_{H|_T^{-1}(y)}(1)\right),
    \end{equation*}
    (here, we used that $H|_T = f|_T$ where $f$ is a distinguished map).
    According to Lemma~\ref{lem:Hpath},
    \begin{equation*}
        H\left(\gamma_{H|_T^{-1}(y)}(1)\right) =
        e^{\int_{\gamma_x} \eta} y, \quad
        \text{with} \quad x := H|_T^{-1}(y).
    \end{equation*}
    Since $T$ is tangent to $\ker\eta$, by concatenating
    $\gamma_x$ with the image of the segment
    $[H(\gamma_x(0)),H(\gamma_x(1))]$ under $H|_T^{-1}$,
    one gets a loop $\tilde{\gamma}_x$ such that
    $\int_{\tilde{\gamma}_x} \eta = \int_{\gamma_x} \eta$.
    Since $\gamma_x$ is $C^0$-close to $\gamma$,
    one can reparametrize $\tilde{\gamma}_x$ such
    that this loop is $C^0$-close to $\gamma$,
    so $\tilde{\gamma}_x$ is homotopic to
    $\gamma$
    and
    the conclusion follows.
\end{proof}

A consequence is the following (see \cite[\S 2.5]{Haefliger}).

  \begin{corollary}
    Let $\Sigma\subset \{ H=0\}$ be a connected component of $\{ H=0\}$
    without critical point of $H$.
    If the pull-back of the Lee form to $\Sigma$ is not trivial,
    there exist leafs of $\cF$ different from $\Sigma$, the closure of
    which contains $\Sigma$.
\end{corollary}
 Examples \ref{ssaxaconsdim2} and \ref{sstransLee} show that a non-compact leaf can go far away from $\{ H=0\}$.

\subsection{Invariance and isotropy}\label{ssisotropy}

\begin{proposition}\label{prop:isotropicinv}
    Let $L$ be a submanifold of $M$.
    \begin{enumerate}
        \item If $L$ is isotropic and invariant,
            then it is tangent to $\cF$.
        \item If $L$ is coisotropic and tangent to $\cF$,
            then it is invariant.
        \item If $L$ is invariant and tangent to $\cF$,
            then the pull-back of $\omega$ to $L$ is degenerate.
            In particular, if $L$ is of even dimension $2k$,
            the pull-back of $\omega^k$ to $L$ is zero.
            An invariant surface tangent to $\cF$ is thus isotropic.
    \end{enumerate}
\end{proposition}

\begin{proof}
    If $L$ is an invariant isotropic submanifold,
    then $X$ is tangent to $L$ so
    $\forall v\in TL$, $\ud_\eta H\cdot v = \omega(X,v) = 0$.
    Conversely, if $L$ is coisotropic and tangent to $\cF$,
    for $x\in L$, $T_xL\supset (T_xL)^\omega \supset (\cF_x)^\omega
    = \R X(x)$ so $L$ is invariant.
    If $i:L\hookrightarrow M$ is invariant and tangent to $\cF$,
    then $X|_L$ is in the kernel of $i^*\omega$.
\end{proof}

Combining Proposition~\ref{prop:isotropicinv} and Corollary~\ref{cor:01law},
one gets the following result.

\begin{corollary}\label{cor:isotropicinvloop}
    An isotropic invariant submanifold on which the pull-back
    of the Lee form is not exact is included in
    $\{ H= 0\}$.
\end{corollary}

Corollary~\ref{cor:isotropicinvloop} can be applied to
Lagrangian graphs of $T^*_\beta Q$ for a non-exact closed 1-form
$\beta$ of $Q$.
Indeed, for every $\beta$-closed $1$-form $\alpha$ of $Q$, $q\mapsto \alpha_q$ defines
a Lagrangian section $Q\hookrightarrow T^*_\beta Q$
pulling back the Lee form to the non-exact form $\beta$.

We now are interested in how the dynamics can force the isotropy. \\
 Following \cite{Schwartzman}, we recall that a point $x\in M$ is {\em quasi-regular} if for every continuous map $f: M\to \R$, the following limit exists
 $$\lim_{t\to+\infty}\frac{1}{t}\int_0^tf(\varphi_s x)ds.$$
 Then we can associate to every quasi-regular point its {\em asymptotic cycle} $A(x)\in H_1(M, \R)$ that satisfies for every continuous closed 1-form $\nu$ on $M$
 $$\langle [\nu], A(x)\rangle =\lim_{t\to\infty}\frac{1}{t}\int_0^t\nu(X_H\circ \varphi_s(x))ds.$$
 Moreover, if $\mu$ is an invariant Borel probability by $(\varphi_t)$, $\mu$ almost every point is quasi-regular and the  asymptotic cycle $A(\mu)\in H_1(M, \R)$ of $\mu$ is defined by
 $$\langle [\nu], A(\mu)\rangle=\int \langle [\nu], A(x)\rangle d\mu(x).$$
 We have 
 \begin{proposition}\label{Pasymptrota}
 Let $(R_{t\alpha})_{t\in\R}$ the flow of rotations of $\T^n$ with vector $\alpha\in \R^n$ that is defined by 
 $$R_{t\alpha}(\theta)=\theta+t\alpha.$$
We identify $H_1(\T^n)$ with $\R^n$ in the usual way. Then every point of $\T^n$ is quasi-regular, and the asymptotic cycle of every point of $\T^n$ and of every invariant probability measure is $\alpha$.
 
 \end{proposition}
 
 Observe that when two flows $(f_t):M\righttoleftarrow$ and $(g_t):N\righttoleftarrow$ are conjugated via some homeomorphism $h:M\to N$, then the quasi-regular points of $(g_t)$ are the $h$-images of the quasi-regular points of $(f_t)$ and that when $x\in M$ is quasi-regular, we have
 $$h_* A( x)= A(h(x)).$$
 This allows us to introduce a notion of {\em rotational torus} $\cT$ for a flow $(f_t):M\righttoleftarrow$. A rotational torus is a $C^0$-embedded torus $j:\T^m\hookrightarrow M$ such that $(j^{-1}\circ f_t\circ j)$ is a flow of rotation. When $j$ is a $C^1$-embedding, $\cT$ is a {$C^1$-rotational torus}. Thanks to Proposition \ref{Pasymptrota}, all the points of a rotational torus are quasi-regular with the same asymptotic cycle that we denote by $A(\cT)\in H_1(M)$ and every measure with support in $\cT$ has also the same asymptotic cycle.

Let us prove a result that is reminiscent of a result of Herman in the symplectic setting \cite{Herman1989}.

\begin{proposition}\label{propositionherman}Assume that
$j(\T^m)=\cT$ is a $C^1$-rotational torus
for a conformal Hamiltonian flow $(\varphi_t)$ of $(M,\eta,\omega)$.
Then 
\begin{itemize}
    \item if the flow 
     restricted to $\cT$ is minimal, $\omega=\ud_\eta\lambda$ is $\eta$-exact and $j^*\eta$ is exact, then $\cT$ 
    is isotropic;
    \item if 
  $\cT$ is not isotropic, then 
    $\langle [\eta], A(\cT)\rangle=0$.
\end{itemize}
In particular, when the cohomological class of $\eta$ is rational and non zero and when the flow 
 restricted to $\cT$ is minimal then  $\cT$ 
is isotropic.
\end{proposition}
\begin{proof}
 We use the notation $R_{t\alpha}=j^{-1}\circ \varphi_t\circ j$.\\
Let us prove the first point. As $j^*\eta=\ud f$ is exact, we have 
$$\ud(e^{-f}j^*\lambda)=e^{-f}(j^*\ud\lambda-\ud f\wedge j^*\lambda)=e^{-f}j^*(\ud_\eta\lambda)=e^{-f}j^*\omega.$$
Hence $e^{-f}j^*\omega$ is exact. Observe that $j^*X_H=\alpha$. Hence
$\forall x\in\T^m$, $\forall t\in\R$,
$$r_t(j(x))=\int_0^t\eta(\varphi_s(j(x)))X_H(\varphi_s(j(x)))\ud s=\int_0^t\ud f(R_{s\alpha}(x))\alpha \ud s=f(x+t\alpha)-f(x).
$$
Because $(\varphi_t^H)^*\omega=e^{r_t}\omega$ and $\varphi_t^H\circ j=j\circ R_{t\alpha}$, we deduce
$$R_{t\alpha}^*(j^*\omega)=e^{r_t\circ j}j^*\omega
$$
and then 
$$R_{t\alpha}^*(e^{-f}j^*\omega)=e^{-f}j^*\omega.
$$
If we write $e^{-f}j^*\omega=\sum_{1\leq i<j\leq m}a_{i, j}\ud x_i\wedge \ud x_j$, we deduce that every continuous function $a_{i, j}$ is invariant by $(R_{t\alpha})$, and then constant because the flow is minimal. The form $e^{-f}j^*\omega$ is constant and exact, it is then the zero form and $\cT$  
is isotropic.

    Let us prove the second point. We assume that $\cT$ 
    is not isotropic. Hence $j^*\omega$ is not zero. There exists a sequence $(t_n)$ of real numbers that tends to $+\infty$ and satisfies 
    $$\lim_{n\to\infty}t_n\alpha=0\quad{\rm in}\quad \T^m.$$
    Then $(R_{t_n\alpha})$ tends to $\id_{\T^m}$ in topology $C^1$. We deduce that $(R_{t_n\alpha}^*(j^*\omega))$ tends to $j^*\omega$.
    Moreover, we also have
    $$R_{t_n\alpha}^*(j^*\omega)=e^{r_{t_n}\circ j}j^*\omega $$
    where, as $\varphi_s\circ j = j\circ R_{s\alpha}$
    (so $X\circ j = \ud j\cdot\alpha$),
    $$r_{t_n}\circ j(x)=\int_0^{t_n}\eta(\ud j(R_{s\alpha}(x))\alpha)\ud s
    =t_n\langle [j^*\eta],\alpha\rangle +o(1).$$
    As $j^*\omega$ is not the zero form, we have then 
   $ 0=\lim_{n\to\infty}r_{t_n}(x)$ which implies that $[j^*\eta]$ is orthogonal to $\alpha$   i.e.  
    $\langle [\eta], A(\cT)\rangle=0$.
\end{proof}
If we assume that the invariant torus is $C^3$, we can relax the hypothesis on the dynamics for the second point of Proposition \ref{propositionherman}, asking only a $C^0$-conjugacy. The main argument that we use is very similar to part 3 of \cite{ArnaudFejoz}.\\

\begin{proposition}\label{PropositionYomdin}
Assume that $\cT\subset M$ is a  
$C^3$-submanifold of
$M$   that is a $C^0$-rotational torus 
of a conformal Hamiltonian flow $(\varphi_t)$ of $(M,\eta,\omega)$ such that $\langle [\eta],
A(\cT)\rangle\neq0$.
Then $\cT$ is isotropic. 

In particular, when the cohomological class of $\eta$ is rational and non zero and when the flow   restricted to $\cT$ is minimal,  
then $\cT$ is isotropic.
\end{proposition}
\begin{proof}[Proof of Proposition \ref{PropositionYomdin}]
    We denote the
    canonical injection $\cT\hookrightarrow M$ by $j$. We endow $\cT$ with a
    Riemannian metric and
    denote by $d_\cG$ the distance along the leaves of the characteristic
    foliation $\cG$ of $j^*\omega.$ We  assume that $\cT$ is not isotropic. We
    denote the maximum rank of $j^*\omega$ by $r$ and by $U$ the open set 
$$
U=\{ x\in \cT\ |\  \rank (j^*\omega(x))=r\}
$$
As $\varphi_t^*(j^*\omega)=e^{r_t}j^*\omega$, this set is invariant by the flow. \\
A result of   Proposition \ref{Pasymptrota} 
is
\begin{equation}\label{ERN} r_{n}(x)=n\left(\langle [\eta],
  A(\cT)\rangle+o_{n\to\infty}(1)\right),\quad
\forall x\in\cT.
\end{equation}

There are two cases:
\begin{itemize}
    \item either $U=\cT$ is compact; we choose $x\in U$ and $\cK:=U$;
    \item or $U\neq\cT$. Then the closure of the orbit of a fixed point $x\in
        U$ is homeomorphic to a torus with dimension $k<m$. As
        $(\varphi_t|_\cT)$ is conjugate to a flow of rotation, there exists
        a compact invariant neighbourhood $\cK$ of $x$ in $U$ that is
        homeomorphic to $\T^k\times [-1, 1]^{m-k}$. 
\end{itemize}
In $\cK$, we consider the characteristic foliation $\cG$ of $\omega$.   
We now follow the arguments and notation of \cite{ArnaudFejoz}
(except that $\cF$ and the $\cF_i$'s are here denoted
$\cG$ and $\cG_i$).
We use a
finite covering of $\cK$ by foliated charts $\cW_1, \dots, \cW_I$ in $\cU$ and
denote by $\cG_i$ the foliation restricted to $\cW_i$. Then there exists a
constant $\mu>0$ such that every $(m-r)$-submanifold $\cS$ of $\cW_i$ that
intersects every leaf of $\cG_i$ at most once satisfies
$|\omega^\frac{r}{2}(\cS)|\leq \mu$. \\
Moreover, we may assume that there exists $\varepsilon>0$ such that:
\begin{enumerate}[(i)]
    \item\label{it:1}
if $x, y$ are in some $\cW_i$,  and such  that $d_\cG(x, y)<\varepsilon$,
then $x$ and $y$ are in the same leaf of $\cW_i$
\end{enumerate}
  where $d_\cG$ is the distance along the leaves.\\
We also have the existence of $\nu\in (0, \varepsilon)$ such that 
\begin{equation}\label{Edistfeuille}
    d_\cG(x,y)<\nu\Rightarrow d_\cG(\varphi_{-1}(x),
\varphi_{-1}(y))<\varepsilon,\quad
\forall x,y\in\cK.
\end{equation}

We then use a decomposition $(Q_j)_{1\leq j\leq J}$ of $\cK$ into submanifolds
with corners that may intersect only along their boundary such that every $Q_j$
is contained in at least one $\cW_i$ that satisfies:
\begin{enumerate}[(i)]
    \setcounter{enumi}{1}
    \item\label{it:2}
 if $Q_j\subset \cW_i$, then if $x, y\in Q_j$ are in the same leaf of $\cW_i$,
 we have $d_{\cG}(x, y)<\nu$.
\end{enumerate}
 If $\cS$ is a piece of $r$-dimensional submanifold   contained in  some $Q_{j_0}\subset
 \cW_{i_0}$ that is transverse to $\cG$ and intersects every leaf of
 $\cG_{i_0}$ at most once, let us consider $\cS'=\varphi_1(\cS\cap
 Q_{j_0})\cap Q_{j_1}$ for some $j_1$. Then $\cS'$ is also transverse to $\cG$.
 Let $\cW_{i_1}$ that contains $Q_{j_1}$ and let us assume that $x, y\in\cS'$
 are in a same leaf of $\cG_{i_1}$. Because of (\ref{it:2}) and
 \eqref{Edistfeuille},
 $d_\cG(\varphi_{-1}(x), \varphi_{-1}(y))<\varepsilon$ and by (\ref{it:1}), we
 have $\varphi_{-1}(x)=\varphi_{-1}(y)$ and $x=y$. Iterating this
 argument, we deduce that all the sets 
 $$\cS'=\varphi_k(\cS\cap Q_{j_0})\cap \varphi_{k-1}(Q_{j_1})\cap \dots \cap Q_{j_k}$$
 are such that if $Q_{j_k}\subset \cW_{i_k}$, $\cS'$ intersects every leaf of $\cF_{i_k}$ at most once and then $|\omega^\frac{r}{2}(\cS')|\leq\mu$. \\
  If now $N_k$ is the number of $k$-uples $(j_1, \dots, j_k)$ such that
  $\varphi_k(\cS\cap Q_{j_0})\cap \varphi_{k-1}(Q_{j_1})\cap \dots \cap
  Q_{j_k}\neq\emptyset$, then we have
  \begin{equation}\label{Einegomega}|\omega^\frac{r}{2}(\varphi_{k}(\cS))|\leq N_k\mu.\end{equation}
 By \eqref{ERN}, we have 
 \begin{equation}\label{Ecroissanceomega}
   \omega^\frac{r}{2}(\varphi_{k}(\cS))=\exp\Big(k  \frac{r}{2}\big(\langle [\eta],  A(\cT)\rangle+o_{k\to\infty}(1)\big)\Big)\omega^\frac{r}{2}(\cS).
\end{equation}
Combining \eqref{Einegomega} and \eqref{Ecroissanceomega}, we deduce that 
$$\limsup_{k\to\infty}\frac{1}{k}\log(N_k)\geq |\langle [\eta],   A(\cT)\rangle|>0$$
is a lower bound for the topological entropy of $(\varphi_t|_\cT)$. But this
contradicts the fact that $(\varphi_t|_\cT)$ is $C^0$-conjugate to a flow of
rotation and has zero entropy.
\end{proof}

\appendix

\section{Isotropic submanifolds}
\label{se:isotropic}

\subsection{Isotropic embeddings}

\begin{lemma}\label{lem:legendrianEmb}
    Given a manifold $N$ endowed with a closed $1$-form $\beta$,
    there exists a Legendrian embedding of $N$ in a contact manifold
    $(V,\alpha)$ endowed with a closed $1$-form, the
    pull-back to $N$ of which is $\beta$.
    This contact manifold is closed if $N$ is closed.
\end{lemma}

\begin{proof}
    Given a submanifold $M'$ of
    a Riemannian manifold $(M,g)$, we denote $\nu^1M'\subset T^1M$
    its unit normal bundle.
    Let us endow $N$ and $S^1$ with Riemannian metrics
    and let us consider unit tangent bundle $V$ of
    the product Riemannian manifold
    $(N\times S^1,g)$ endowed with its standard contact form.
    Let us recall that unit normal bundles of submanifolds of
    $N\times S^1$ are Legendrian submanifolds of $V$.
    Therefore, $\nu^1(N\times\{ x\})$, $x\in S^1$ fixed,
    is a Legendrian submanifold, it is the disjoint union of two copies
    of $N$.
    We lift the closed $1$-form $\beta$ to $V$ by pulling it back
    by the canonical projection $T^1(N\times S^1)\to N$.
\end{proof}

\begin{lemma}\label{lem:lagrangianEmb}
    Given a manifold $N$ endowed with a closed $1$-form $\beta$,
    there exists a Lagrangian embedding of $N\times S^1$
    in a conformal exact symplectic manifold $(M,\eta,\omega)$
    such that the pull-back of $\eta$ to $N\times S^1$
    is $\pi^*\beta-\ud\theta$ where
    $\pi$ is the projection on the first factor.
    This manifold $M$ is closed if $N$ is closed.
\end{lemma}

\begin{proof}
    Let $(V,\alpha)$ be the contact manifold of
    Lemma~\ref{lem:legendrianEmb}, we assume $N\subset V$
    and identify $\beta$ with its pull-back to $V$.
    Then the $\beta$-twisted symplectization $(M,\eta,\omega)$
    of $V$ endowed with its standard Lee form satisfies
    the statement.
\end{proof}

\subsection{Weinstein neighborhood of
isotropic submanifolds}
\label{se:weinstein}

We extend the usual Weinstein neighborhood theorem
\cite[Lecture~5]{Weinstein1977}
to the conformal setting
following and adapting \cite[Section~2.5.2]{Geiges2008}.
The special case of Lagrangian submanifolds was
already treated by Chantraine-Murphy \cite[Theorem~2.11]{ChantraineMurphy2019}.

Let us first describe the conformal structure of the
local model of a neighborhood of an isotropic manifold $Q^k\subset
(M^{2n},\eta,\omega)$.
Let us denote $T_Q M$ the restriction of the tangent bundle of $M$ to $Q$
and $TQ^\omega\subset T_Q M$ the $\omega$-orthogonal bundle of $TQ$.
Then the normal bundle $\pi:\nu Q\to Q$ can be non-canonically decomposed as
\begin{equation}\label{eq:decnuQ}
    \nu Q = T_Q M/TQ \simeq TQ^\omega/TQ \oplus T_Q M/TQ^\omega.
\end{equation}
In order to fix this decomposition, let us fix a complex
structure $J$ compatible with $\omega$,
\emph{i.e.} such that $g:=\omega(\cdot,J\cdot)$ defines a Riemannian
metric.
With respect to $g$, $\nu Q$ is canonically isomorphic to
the orthogonal vector bundle
$TQ^\bot$, $TQ^\omega /TQ$ is isomorphic to $(TQ\oplus J(TQ))^\omega$
and $T_Q M/TQ^\omega$ to $J(TQ)$ so that the decomposition
(\ref{eq:decnuQ}) takes the concrete form
\begin{equation*}
    TQ^\bot = (TQ\oplus J(TQ))^\omega \oplus J(TQ).
\end{equation*}
We will work through this identification.
Fibers of $TQ^\omega/TQ$ are symplectic vector spaces of dimension $2(n-k)$
for the structure induced by $\omega$.
Let $\beta$ be the pull-back of the Lee form to $Q$.
The fiber bundle $T_QM/TQ^\omega$ is diffeomorphic to $T^*_\beta Q$
under $(q,[v])\mapsto \omega_q(v,\cdot)$.
One can check in local coordinates that the direct sum of the fibered symplectic
form and the conformal symplectic form of $T^*_\beta Q$ defines
a $\pi^*\beta$-conformal symplectic form $\omega_{\nu Q}$ such that
$(\nu Q,\eta_{\nu Q}=\pi^*\beta,\omega_{\nu Q})$
is a well-defined conformal symplectic manifold.
Moreover, for all $q\in Q$, $(\omega_{\nu Q})_{(q,0)}=\omega_q$ by definition:
both structures agree along $Q$
(seeing $\nu Q$ as $TQ^\bot$, we recall that
$T_{(q,0)}(TQ^\bot)$ is naturally identified with $T_q Q$).
The structure of $\nu Q$ depends on the isotropic embedding $Q\subset M$,
when $Q$ is Lagrangian $\nu Q\simeq T^*_\beta Q$ only depends on $\eta|_Q$.

\begin{theorem}[Weinstein neighborhood]\label{thm:weinstein}
    Let $Q\subset (M,\eta,\omega)$ be an isotropic submanifold,
    there exist a neighborhood $U_1\subset M$ of $Q$, a neighborhood
    $U_2\subset \nu Q$ of the $0$-section and a conformal symplectomorphism
    $\varphi : (U_1,\eta,\omega)\to (U_2,\eta_{\nu Q},\omega_{\nu Q})$ 
    sending $Q$ to the $0$-section canonically.
\end{theorem}

The proof is an adaptation of the symplectic case.
We will use the following stability theorem proven by Chantraine-Murphy.

\begin{theorem}[{\cite[Theorem~2.10]{ChantraineMurphy2019}}]\label{thm:stability}
    Let $(M,\eta)$ be a closed manifold endowed with a closed $1$-form
    and let $(\omega_t)_{t\in[0,1]}$ be a path of $\eta$-conformal symplectic
    forms such that $\omega_t = \omega_0 +\ud_\eta\lambda_t$.
    There exists an isotopy $\varphi_t:M\to M$ and functions
    $f_t:M\to\R$, $t\in[0,1]$, such that $\varphi_t^*\eta=\eta+\ud f_t$
    and $\varphi_t^*\omega_t = e^{f_t}\omega_0$.
\end{theorem}

Let us first prove the following conformal extension of
\cite[Lemma~3.14]{McDuffSalamon1998}.

\begin{lemma}\label{lem:extension}
    Let $(M,\eta)$ be a manifold endowed with a closed $1$-form
    and $Q\subset M$ be a compact submanifold.
    Let us assume that there exist two $\eta$-conformal symplectic
    forms $\omega_0$ and $\omega_1$ agreeing on $T_qM$ for all $q\in Q$.
    There exist two neighborhoods $U_0$ and $U_1$ of $Q$, a diffeomorphism
    $\psi:U_0\to U_1$ and a map $f:U_0\to\R$ vanishing on $Q$ such that
    $\psi|_Q = \id$, $\psi^*\omega_1 = e^f\omega_0$ and $\psi^*\eta=\eta+\ud f$.
\end{lemma}

\begin{proof}[Proof of Lemma~\ref{lem:extension}]
    Let us endow $M$ with a Riemannian metric and let us define a
    tubular neighborhood $U$ of $Q$ as the image under the diffeomorphism
    $(q,v)\mapsto \exp_q(v)$ of a neighborhood of the $0$-section of the normal
    bundle of $Q$.
    Let $\pi:U\to Q$ be the orthogonal projection.
    Since $\pi$ is a retraction by deformation, $\eta|_U$ is cohomologous
    to $\pi^*\beta$ where $\beta$ is the pull-back of $\eta$ to $Q$.
    One can assume $\eta = \pi^*\beta$:
    indeed, $\eta = \pi^*\beta +\ud g$ with $g:U\to\R$ vanishing on $Q$
    so one can do a gauge transformation.

    The remainder of the proof closely follows the non-conformal one
    \cite[Lemma~3.14]{McDuffSalamon1998}.
    We set $\tau:=\omega_1-\omega_0$ in order to show that $\tau=\ud_\eta\sigma$
    for some $\sigma$ and apply Theorem~\ref{thm:stability} with $\lambda_t =
    t\sigma$. In order to prove that, we consider the map
    $\phi_t : \exp_q(v)\mapsto\exp_q(tv)$ defined on $U$ for $t\in[0,1]$
    so that $\phi_0=\pi$, $\phi_1=\id$ and $\tau=\phi_1^*\tau-\phi_0^*\omega$
    ($\pi^*\omega_0 = \pi^*\omega_1$ by assumption).
    Let us set $X_t = \partial_t\phi_t\circ\phi_t^{-1}$,
    well-defined for $t\in (0,1]$. Then
    \begin{equation*}
        \frac{\ud}{\ud t}(\phi_t^*\tau) = \phi_t^*(
        \ud_\eta(\iota_{X_t}\tau) + \pi^*\beta(X_t)\tau)
        = \phi_t^*(\ud_\eta(\iota_{X_t}\tau)),
    \end{equation*}
    since $\ud\pi\cdot X_t = 0$.
    We remark that $\sigma_t := \phi_t^*(\iota_{X_t}\tau)$ is smoothly
    defined for $t\in[0,1]$ ($X_0$ is not well-defined but
    $X_0\circ\phi_0$ is). We conclude by setting $\sigma = \int_0^1\sigma_t\ud t$.
\end{proof}

\begin{proof}[Proof of Theorem~\ref{thm:weinstein}]
    Let us restrict ourselves to a tubular neighborhood $V\simeq \nu Q$
    of $Q\subset M$.
    Under the gauge change $\eta = \pi^*\beta$ for the structure
    $(\eta,\omega)|_V$, we are under the assumptions of Lemma~\ref{lem:extension}
    with $(V,\pi^*\beta)$, $(\eta,\omega)|_V$, and $(\eta_{\nu Q},\omega_{\nu Q})$.
\end{proof}

\bibliography{_biblio}
\bibliographystyle{amsplain}

\end{document}